\documentclass[12pt, a4paper]{article}

\usepackage[utf8]{inputenc}
\usepackage[english]{babel}

\usepackage{amsmath,amssymb,amsfonts,mathrsfs}

\usepackage[amsmath,thmmarks]{ntheorem}

\numberwithin{equation}{section}

\newtheorem{thm}{Theorem}[section]

\newtheorem{lem}[thm]{Lemma}
\newtheorem{prop}[thm]{Proposition}
\newtheorem{rem}[thm]{Remark}
\newtheorem{defn}[thm]{Definition}
\newtheorem{axm}[thm]{Axiom}
\newtheorem{exl}[thm]{Example}

\newtheorem{notn}[thm]{Notation}
\newtheorem{conv}[thm]{Convention}

\theoremstyle{empty}

\theoremstyle{nonumberplain}
\theorembodyfont{\normalfont}
\theoremsymbol{\ensuremath{\square}}
\newtheorem{proof}{Proof}

\usepackage{hyperref}

\title{On a Completion of Cohomological Functors Generalising Tate Cohomology II}
\author{Max Gheorghiu}
\date{1 April 2026}

\usepackage[margin=2.5cm]{geometry}

\usepackage{tikz-cd}
\usepackage{stmaryrd}

\newcommand{\triangleleftneq}{%
  \mathrel{\ooalign{$\lneq$\cr\raise.22ex\hbox{$\lhd$}\cr}}}

\begin{document}

\setlength{\parindent}{0cm}

\maketitle

\begin{abstract}
Viewing group cohomology as a cohomological functor, G.\! Mislin has generalised Tate cohomology from finite groups to all discrete groups by defining a completion for cohomological functors in 1994. In a previous paper, we have constructed for a cohomological functor $T^{\bullet}: \mathcal{C} \rightarrow \mathcal{D}$ its Mislin completion $\widehat{T}^{\bullet}: \mathcal{C} \rightarrow \mathcal{D}$ under mild assumptions on the abelian categories $\mathcal{C}$ and $\mathcal{D}$, which generalises Tate cohomology to all $T1$ topological groups. In this paper, we investigate the properties of Mislin completions. As their main feature, Mislin completions of Ext-functors detect finite projective dimension of objects in the domain category. We establish a version of dimension shifting, an Eckmann--Shapiro result as well as cohomology products such as external products, cup products and Yoneda products.
\end{abstract}

\tableofcontents

\section{Introduction}\label{sec:intro}

\subsection{A generalisation of Tate cohomology and its properties}\label{subsec:expo}

Originally, Tate cohomology was developed only for finite groups by J.\! Tate in the paper~\cite{tat52} from 1952 motivated by class field theory. As it unites group homology and group cohomology of finite groups in a convenient manner, it attracted the attention of group theorists~\cite[p.~128]{bro82}. It was generalised to groups of finite virtual cohomological dimension in F.\! T.\! Farrell's paper~\cite{far77} from 1977 and then to all groups in D.\! J.\! Benson and J.\! F.\! Carlson's paper~\cite{ben92} from 1992, F.\! Goichot's paper~\cite{goi92} from 1992 and G.\! Mislin's paper~\cite{mis94} from 1994. The different approaches of the above papers have been used in generalised constructions by several authors with applications to group theory, ring theory and homotopical algebra. There is a need for a uniform account that explains why the approaches underlying these constructions all lead to the same conclusions: a theory that does not only work for discrete groups or modules over a ring, but also in greater generality. The subject of the previous paper~\cite{ghe24} was to provide a uniform treatment of the entire theory, putting in place detailed proofs that establish the equivalence of different definitions of the different authors. \\

This is the follow-up paper to~\cite{ghe24} in which we investigate the properties of this generalisation of Tate cohomology that we describe in detail further below. In summary, our theory detects whether a group (object) has finite cohomological dimension. It satisfies a form of dimension shifting, meaning that one can express a Tate cohomology group of any degree isomorphically in terms of any other degree. It satisfies an Eckmann–Shapiro Lemma, meaning that Tate cohomology of finite index (closed) subgroups of a discrete or profinite group can be described by Tate cohomology of the entire group. Cup products and Yoneda products are developed in Tate cohomology under certain conditions and their key properties are established.

\subsection{Outline of the general and uniform theory}\label{subsec:outline}

We outline in the following our uniform theory of Tate cohomology. The key notion is a completion of cohomological functors. Connecting homomorphisms are one of the most important structures for cohomology. Particularly in group cohomology, they constitute fundamental computational tools. Given the prominence of connecting homomorphisms, one can view cohomological functors as a generalisation of group cohomology only preserving the connecting homomorphisms. More specifically, if $\mathcal{C}$, $\mathcal{D}$ are abelian categories, then a family of additive functors $(T^n: \mathcal{C} \rightarrow \mathcal{D})_{n \in \mathbb{Z}}$ is a cohomological functor if there are connecting homomorphisms $(\delta^n: T^n \rightarrow T^{n+1})_{n \in \mathbb{Z}}$ satisfying two straightforward axioms~\cite[p.~201--202]{kro95}. In particular, if $G$ is a discrete group, $R$ a discrete ring and $A$ a  discrete $R$-module, then  setting $H_R^n(G, -) = 0$ and $\mathrm{Ext}_R^n(A,-) = 0$ for $n < 0$ renders group cohomology and Ext-functors into cohomological functors~\cite[p.~201]{kro95}, \cite[p.~295]{mis94}. Our generalisation of Tate cohomology in~\cite{ghe24} is obtained over a completion of cohomological functors that we term a Mislin completion. More specifically, a Mislin completion of a cohomological functor $T^{\bullet}: \mathcal{C} \rightarrow \mathcal{D}$ is another cohomological functor $\widehat{T}^{\bullet}: \mathcal{C} \rightarrow \mathcal{D}$ together with a morphism $\Phi^{\bullet}: T^{\bullet} \rightarrow \widehat{T}^{\bullet}$ such that $\widehat{T}^n(P) = 0$ for any projective $P \in \mathrm{obj}(\mathcal{C})$ and $n \in \mathbb{Z}$. It satisfies the universal property that any morphism $T^{\bullet} \rightarrow V^{\bullet}$ to a cohomological functor $V^{\bullet}: \mathcal{C} \rightarrow \mathcal{D}$ also vanishing on projectives factors uniquely through $\Phi^{\bullet}$. By its universal property, any Mislin completion is unique up to isomorphism. This completion of cohomological functors tracing back to G.\! Mislin's paper~\cite{mis94} bears the advantage that it relates group cohomology with Tate cohomology in a unique manner. \\

Let us detail under what mild assumptions our Mislin completions exist and how they compare to the literature. By Theorem~3.2 in~\cite{ghe24}, a Mislin completion of a cohomological functor $T^{\bullet}: \mathcal{C} \rightarrow \mathcal{D}$ exists whenever $\mathcal{C}$ has enough projectives and in $\mathcal{D}$ all countable direct limits exist and are exact. Throughout the paper, we assume that the abelian categories $\mathcal{C}$ and $\mathcal{D}$ satisfy the above assumptions. Mislin completions generalise F.\! T.\! Farrell's approach from~\cite{far77} as we demonstrate in Section~\ref{sec:complextandcan} . F.\! T.\! Farrell formulated it originally for groups of finite virtual cohomological dimension where P.\! Symonds extended it to profinite groups of finite virtual cohomological dimension in~\cite[p.~34]{sym07}. Our theory applies to all groups and to all profinite groups. L.\! L.\! Avramov and O.\! Veliche defined Mislin completions for Ext-functors of modules over a commutative Noetherian ring $R$ and used them to characterise when $R$ is regular, complete intersection or Gorenstein. Our theory encompasses their ring theoretic work. The generality of our theory prompts us to introduce the following terminology. Whenever we apply our theory it to groups or to topological groups, we follow the convention from~\cite{kro95} and term the resulting generalisation complete cohomology. Analogously, we term the generalised Ext-functors resulting from our theory completed Ext-functors. \\

As a remarkable feature, our generalisation of Tate cohomology is applicable to condensed mathematics and thus to most topological groups of interest. Condensed mathematics a powerful novel theory developed by D.\! Clausen and P.\! Scholze in 2018~\cite{sch19}. In a nutshell, it provides a unified approach for studying topological groups, rings and modules~\cite[p.~6]{sch19}. In~\cite{sch20}, P.\! Scholze writes that he wants ``to make the strong claim that in the foundations of mathematics, one should replace topological spaces with condensed sets''. Here, a condensed set can be thought as a sheaf of sets. Let us substantiate P.\! Scholze's claim. In practice, most topological spaces (resp.\! groups, rings, etc) are $T1$, meaning that all their points are closed. These can be functorially turned into condensed sets (resp.\! groups, rings, etc), meaning that one can work only with the corresponding condensed object. Condensed mathematics provides an example of a Milsin completion that cannot be treated through any previous framework from the literature. If $\mathcal{R}$ is a condensed ring, $\mathrm{Cond}(\mathrm{Mod}(\mathcal{R}))$ the category of condensed $\mathcal{R}$-modules and $\mathrm{Cond}(\mathbf{Ab})$ the category of condensed abelian groups, then the  Mislin completion of the enriched Ext-functor
\[ \underline{\widehat{\mathrm{Ext}}}_{\mathcal{R}}^{\bullet}(A, -): \mathrm{Cond}(Mod(\mathcal{R})) \rightarrow \mathrm{Cond}(\mathbf{Ab}) \]
can be constructed only by means of our uniform theory~\cite[Remark~7.2]{ghe24}. One can define complete cohomology for all condensed groups and thus, for all $T1$ topological groups, meaning for most topological groups as the following list demonstrates \cite[Theorem~7.1]{ghe24}.
\begin{itemize}
\item Complete cohomology can be defined for all Lie groups and in particular, for $p$-adic analytic groups.
\item Complete cohomology can be defined for any Galois group or equivalently, for any profinite group. More specifically, a profinite group is an inverse limit of finite discrete groups, meaning that it can be assembled from finite groups in a particular manner. Profinite groups are of interest because they possess properties of infinite groups such as finite generation, but also properties of finite groups such as Sylow subgroups. Moreover, they have recently featured in $3$-manifold topology via the study of profinite rigidity of their fundamental groups.
\item Complete cohomology can be defined for totally disconnected locally compact (tdlc) groups that are becoming increasingly fashionable. Tdlc groups can be defined as locally profinite groups, thus they generalise both discrete and profinite groups. One can interpret their profinite open subgroups as taking the role of finite subgroups in discrete groups. Tdlc groups play a crucial role in understanding the structure of the more general locally compact groups. Examples of tdlc groups also include $p$-adic analytic groups and automorphisms of locally finite trees with the compact-open topology.
\item However, complete cohomology cannot be defined for algebraic groups via condensed mathematics. Namely, any topological group can be turned into a condensed group because any cartesian product of two topological spaces results in a product of their corresponding condensed sets. In contrast, a product of two algebraic varieties with the Zariski topology does not result in a product of the corresponding condensed sets.
\end{itemize}

\subsection{Characterising finite cohomological dimension}\label{subsec:findim}

The most prominent feature of complete cohomology is that it detects finite cohomological dimension of a group (object). Here, cohomological dimension is one of the fundamental homological invariants of a group. An object in a category has finite projective dimension if it admits a projective resolution of finite length~\cite[p.~152]{bro82}. A group (object) $G$ has finite cohomological dimension over a ring (object) $R$ if $R$ as a module object over a group ring has finite projective dimension~\cite[p.~184--185]{bro82}. Since the below result applies to condensed mathematics, complete cohomology vanishes for all $T1$ topological groups of finite cohomological dimension.

\begin{lem} (= Lemma~\ref{lem:prevanishing})
\begin{enumerate}
    \item Assume that $T^{\bullet}: \mathcal{C} \rightarrow \mathcal{D}$ is a cohomological functor where $\mathcal{C}$ has enough projectives and in $\mathcal{D}$ all countable direct limits exist and are exact. If $M \in \mathrm{obj}(\mathcal{C})$ has finite projective dimension, then $\widehat{T}^n(M) = 0$ for every $n \in \mathbb{Z}$. In particular, if every object in $\mathcal{C}$ has finite projective dimension such as in a category of modules over a ring of finite global dimension, then $\widehat{T}^{\bullet} = 0$ for any cohomological functor $T^{\bullet}$.
    \item If one takes \textbf{enriched} Ext-functors $\mathrm{Ext}_{\mathcal{C}}^n(A, -): \mathcal{C} \rightarrow \mathcal{D}$ with $A \in \mathrm{obj}(\mathcal{C})$ of finite projective dimension, then $\widehat{\mathrm{Ext}}_{\mathcal{C}}^n(A, -) = 0$ for every $n \in \mathbb{Z}$. In particular, complete cohomology $\widehat{H}_R^{\bullet}(G, M)$ vanishes if the group object $G$ has finite cohomological dimension or the module object $M$ has finite projective dimension. As this applies to any condensed group $G$, this holds for any $T1$ topological group.
\end{enumerate}
\end{lem}

The most prominent feature of completed Ext-functors is that they determine whether an object has finite projective dimension and thus, whether a (topological) group has finite cohomological dimension.

\begin{thm}\label{thm:vanishing}
(= Theorem~\ref{thm:realvanishing}) If $\widehat{\mathrm{Ext}}_{\mathcal{C}}^{\bullet}(A, -): \mathcal{C} \rightarrow \mathbf{Ab}$ denote completed (unenriched) Ext-functors for $A \in \mathrm{obj}(\mathcal{C})$, then the following are equivalent.
\begin{enumerate}
\item The object $A$ has finite projective dimension.
\item $\widehat{\mathrm{Ext}}_{\mathcal{C}}^n(A, -) = \widehat{\mathrm{Ext}}_{\mathcal{C}}^n(-, A) = 0$ for any $n \in \mathbb{Z}$.
\item $\widehat{\mathrm{Ext}}_{\mathcal{C}}^0(A, A) = 0$.
\end{enumerate}
In particular, the zeroeth complete cohomology group detects whether a group (object) has finite cohomological dimension. This applies to any condensed group and thus to any $T1$ topological group.
\end{thm}

This theorem implies that complete cohomology does not vanish for the vast majority groups and thus, constitutes a  nontrivial invariant. More specifically, any group or profinite group with torsion has infinite cohomological dimension and thus, nontrivial complete cohomology.

\subsection{An Eckmann--Shapiro Lemma and dimension shifting}\label{subsec:shapiroshifting}

Complete cohomology of certain subgroups relates to the entire group via an Eckmann--Shapiro Lemma, which can be used to establish dimension shifting. For instance, an Eckmann--Shapiro Lemma pertains to finite index subgroups of discrete groups and open subgroups of profnite groups. Results à la Eckmann--Shapiro are relevant for computational purposes as computing the cohomology for a subgroup might be easier than for the entire group. As this is relevant for such a lemma, a functor $F: \mathcal{C} \rightarrow \mathcal{D}$ is said to preserve projectives if for every projective $P \in \mathrm{obj}(\mathcal{C})$ also $F(P) \in \mathrm{obj}(\mathcal{D})$ is projective.

\begin{lem}[Eckmann--Shapiro]\label{lem:shapiro}
(= Lemma~\ref{lem:eckmannshapiro}) Let $G$ be a group object, $H$ a subgroup object and $R$ a ring object in a category. Denote the abelian category of $R$-module objects with a compatible $G$-action by $Mod_R(G)$ and assume that this category has enough projectives. Assume further that any $M \in \mathrm{obj}(Mod_R(H))$ can be turned into an object in $Mod_R(G)$ by induction $\mathrm{Ind}_H^G(M)$ and coinduction $\mathrm{Coind}_H^G(M)$ while any $M \in \mathrm{obj}(Mod_R(G))$ can be turned into an object in $Mod_R(H)$ by restriction $\mathrm{Res}_H^G(M)$.
\begin{enumerate}
\item If the adjoint functors $\mathrm{Res}_H^G(-)$ and $\mathrm{Coind}_H^G(-)$ are exact and preserve projective objects, then
\[ \widehat{\mathrm{Ext}}_{R, H}^n(\mathrm{Res}_H^G(A), B) \cong \widehat{\mathrm{Ext}}_{R, G}^n(A, \mathrm{Coind}_H^G(B)) \]
as (unenriched) completed Ext-functors for every $n \in \mathbb{Z}$, $A \in \mathrm{obj}(Mod_R(G))$ and $B \in \mathrm{obj}(Mod_R(H))$. If $A = R$, one has
\[ \widehat{H}_R^n(H, B) \cong \widehat{H}_R^n(G, \mathrm{Coind}_H^G(B)) \, . \]
\item If the adjoint functors $\mathrm{Ind}_H^G(-)$ and $\mathrm{Res}_H^G(-)$ are exact and preserve projectives, then
\[ \widehat{\mathrm{Ext}}_{R, G}^n(\mathrm{Ind}_H^G(A), B) \cong \widehat{\mathrm{Ext}}_{R, H}^n(A, \mathrm{Res}_H^G(B)) \]
as (unenriched) completed Ext-functors for every $n \in \mathbb{Z}$, $A \in \mathrm{obj}(Mod_R(H))$ and $B \in \mathrm{obj}(Mod_R(G))$.
\end{enumerate}
\end{lem}

For the below example we note that a profinite ring $R = \varprojlim_{i \in I} R_i$ (respectively, space, module etc.) is defined as an inverse limit of finite discrete rings $R_i$ (respectively, spaces, modules, etc.)~\cite[p.~1]{rib10}. For a profinite group $G = \varprojlim_{j \in J} G_j$ the completed group ring $R{\llbracket}G{\rrbracket} := \varprojlim_{(i, j) \in I \times J} R_i[G_j]$ is a profinite ring that is a profinite version of a (discrete) group ring $R[G]$~\cite[p.~171]{rib10}.

\begin{exl}
(= Example~\ref{exl:eckmannshapiro}) The conditions of the above lemma are satisfied in the following two instances.
\begin{enumerate}
\item $G$ is a discrete group, $H$ a finite index subgroup and $R$ a discrete commutative ring. Modules are taken over the respective group rings.
\item $G$ is a profinite group, $H$ an open subgroup and $R$ a profinite commutative ring. Profinite modules are taken over the respective completed group rings.
\end{enumerate}
\end{exl}

Dimension shifting means that one can express a cohomology group of any degree isomorphically in terms of any other degree. This is especially useful in proofs as proving a statement for any cohomology group becomes equivalent to proving the statement for one specific cohomology group such as the zeroeth. We provide a (partial) version of dimension shifting for Mislin completions.

\begin{thm}
(= Theorem~\ref{thm:dimshifting}) Let $T^{\bullet}: \mathcal{C} \rightarrow \mathcal{D}$ be a cohomological functor where $\mathcal{C}$ has enough projectives and in $\mathcal{D}$ all countable direct limits exist and are exact.
\begin{itemize}
\item For every $M \in \mathrm{obj}(\mathcal{C})$ there is $M^{\ast} \in \mathrm{obj}(\mathcal{C})$ such that $\widehat{T}^{n+1}(M^{\ast}) \cong \widehat{T}^n(M)$ for every $n \in \mathbb{Z}$.
\item If there is a monomorphism $f: M \rightarrow N$ in $\mathcal{C}$ with $\widehat{T}^k(N) = 0$ for every $k \in \mathbb{Z}$, then $\widehat{T}^{n-1}(\mathrm{Coker}(f)) \cong \widehat{T}^n(M)$.
\item Assume that $G$ is a group object, that $R$ is a ring object and that an Eckmann--Shapiro Lemma such as Lemma~\ref{lem:shapiro} holds. Then there is a monomorphism as in the previous assertion for $\widehat{T}^{\bullet} = \widehat{H}_R^{\bullet}(G, -)$ if there exists a subgroup object $H$ of finite cohomological dimension over $R$ such that $\mathrm{Res}_H^G(-)$ is a faithful functor.
\end{itemize}
\end{thm}

\begin{exl}
(= Example~\ref{exl:dimshifting}) The conditions of the third assertion in the above theorem are satisfied in the following two instances.
\begin{enumerate}
\item $G$ is a discrete group, $H$ a finite-index subgroup of finite cohomological dimension and $R$ a discrete commutative ring. Modules are taken to be discrete over the group ring $R[G]$.
\item $G$ is a profinite group, $H$ an open subgroup of finite cohomological dimension and $R$ a profinite commutative ring. Modules are taken to be profinite over the completed group ring $R{\llbracket}G{\rrbracket}$.
\end{enumerate}
\end{exl}

\subsection{Cup products and Yoneda products}\label{subsec:cohomprods}

Most of this paper is dedicated to cup products and Yoneda products for complete cohomology that we construct in great generality and whose properties we establish. Let us first focus on cup products. Thus far, cup products have been developed for Tate--Farrell cohomology of groups with finite virtual cohomological dimension in~\cite[pp.~278--279]{bro82} and their (non-)vanishing in the case of finite groups has been investigated in~\cite{ben92}. The only reason why they are well defined in the latter two settings is because they are obtained from the cup products of `ordinary' cohomology via dimension shifting. Our novel construction of cup products for complete cohomology generalises the previous two constructions by extending them to all groups and to all profinite groups (Lemma~\ref{lem:tatefarrelltwo}). \\

Cup products in complete cohomology descend from external products defined for completed Ext-functors which in turn descend from tensor products. Thus, our treatment of cup products requires the following general preliminaries on tensor products. Tensor products are taken to be bi-additive associate functors. Given our focus on group cohomology, we only consider tensor products $\otimes_R$ in categories of module objects $Mod_R$ over a ring object $R$ or in categories $Mod_R(G)$ of $R$-module objects with a compatible action of a group object $G$. In either case, it is assumed that this category has enough projectives and that the tensor product $P \otimes_R Q$ is projective whenever $P$, $Q$ are projective. Further, it is assumed that there are natural isomorphisms $M \otimes_R R \cong M \cong R \otimes_R M$. For instance, the tensor product of discrete modules and of profinite modules satisfy the above conditions (Example~\ref{exl:tensorprods}).

\begin{thm}
(= Theorem~\ref{thm:extandcupprod}) Let $\otimes_R$ be a tensor product in a category of module objects $Mod_R$ over a ring object $R$ or in a category $Mod_R(G)$ of $R$-module object with a compatible action of a group object $G$. Assume in either case that this category has enough projectives.
\begin{enumerate}
    \item Let $A_{\bullet}$, $C_{\bullet}$ be projective resolutions of module objects $A, C$ such that the tensor product of resolutions $A_{\bullet} \otimes_R C_{\bullet}$ is a projective resolution of $A \otimes_R C$. If $B, E$ are module objects possessing projective resolutions of a specific form, then for every $m, n \in \mathbb{Z}$ external products
    \[ \vee: \widehat{\mathrm{Ext}}_R^m(A, B) \otimes \widehat{\mathrm{Ext}}_R^n(C, E) \rightarrow \widehat{\mathrm{Ext}}_R^{m+n}(A  \otimes_R C, B \otimes_R E) \]
    can be defined for completed (unenriched) Ext-functors.
    \item Assume that the restriction functor $Mod_R(G) \rightarrow Mod_R$ forgetting the $G$-action on $R$-module objects preserves projectives and that $R$ as an object in $Mod_R$ is projective. Then for every $m, n \in \mathbb{Z}$ there are cup products
    \[ \smile: \widehat{H}_R^m(G, M) \otimes \widehat{H}_R^n(G, N) \rightarrow \widehat{H}_R^{m+n}(G, M \otimes_R N) \]
    for complete group cohomology that descend from the above external products.
\end{enumerate}
\end{thm}

\begin{exl}
(= Example~\ref{exl:extandcupprods}) All conditions of the above theorem are satisfied in the following instances.
\begin{enumerate}
    \item $G$ is a discrete group, $R$ a principal ideal domain and the restriction of the $R[G]$-modules $A$, $B$, $C$, $E$ to $R$-modules is projective.
    \item $G$ is a profinite group, $R$ a profinite commutative ring with a unique maximal open ideal and the restriction of the profinite $R{\llbracket}G{\rrbracket}$-modules $A$, $B$, $C$, $E$ to $S$-modules is projective. The $p$-adic integers $\mathbb{Z}_p$ are an example of such a profinite ring where the restriction of any $p$-torsionfree profinite $\mathbb{Z}_p{\llbracket}G{\rrbracket}$-module to a $\mathbb{Z}_p$-module is projective.
\end{enumerate}
\end{exl}

We construct Yoneda products for completed (unenriched) Ext-functors in the greatest possible generality. Recall that Yoneda products of Ext-functors descend from compositions of morphisms in the domain category. In particular, this generalises D.\! J.\! Benson and J.\! F.\! Carlson construction~\cite[p.~110]{ben92}.

\begin{thm}
(= Theorem~\ref{thm:yonedaprods}) Let $\mathcal{C}$ be an abelian category with enough projectives and $F, H, J \in \mathrm{obj}(\mathcal{C})$. If $\otimes$ denotes the tensor product in $\mathbf{Ab}$, then for every $m, n \in \mathbb{Z}$ Yoneda products
    \[ \circ: \widehat{\mathrm{Ext}}_{\mathcal{C}}^n(H, J) \otimes \widehat{\mathrm{Ext}}_{\mathcal{C}}^m(F, H) \rightarrow \widehat{\mathrm{Ext}}_{\mathcal{C}}^{m+n}(F, J) \]
    can be defined for completed (unenriched) Ext-functors.
\end{thm}

Our cohomology products (cup products, external products and Yoneda products) do not only exist in great generality, but also possess almost all properties one would expect. More specifically, they are natural (Lemma~\ref{lem:cohomprodsnatural}) and associative (Lemma~\ref{lem:cohomprodsassoc}). Cup products turn complete cohomology and Yoneda products turn completed Ext-functors into a graded ring with identity (Lemma~\ref{lem:cohomrings}). Lastly, external and cup products satisfy a version of graded commutativity (Proposition~\ref{prop:extprodsgradedcomm}) and are compatible with connecting homomorphisms (Lemma~\ref{lem:cohomprodsandconn}).

\subsection{Detecting finite groups and Tate--Farrell cohomology}\label{subsec:varia}

We close this introduction by highlighting two noteworthy features of complete cohomology. The first feature concerns pro-$p$ groups, which are a well-behaved class of profinite groups. For a fixed prime number $p$, they can be defined as inverse limits of finite $p$-groups and are thus a profinite version of $p$-groups. One example of a pro-$p$ group are the $p$-adic integers $\mathbb{Z}_p$ which are studied in~\cite[Section~1.5]{wil98}. There is a rich theory of the cohomology of pro-$p$ groups. The first feature is that complete cohomology detects finite groups among pro-$p$ groups.

\begin{lem} (= Lemma~\ref{lem:propgroupfinite})
A pro-$p$ group $G$ is finite if and only if
\begin{itemize}
\item $H_{\mathbb{Z}_p}^0(G, -) \ncong \widehat{H}_{\mathbb{Z}_p}^0(G, -)$ and
\item $H_{\mathbb{Z}_p}^n(G, -) \cong \widehat{H}_{\mathbb{Z}_p}^n(G, -)$ for any $n \geq 1$.
\end{itemize}
\end{lem}

As the second feature, completed (unenriched) Ext-functors generalise F.\! T.\! Farrell's approach from~\cite{far77} to the greatest extent. This includes Tate--Farrell cohomology for profinite groups defined in~\cite[p.~34]{sym07} if one takes profinite modules as coefficients (Example~\ref{exl:tatefarrell}). More specifically, a complete resolution of an object $A$ in a category $\mathcal{C}$ is a particular acyclic chain complex $(\overline{A}_n)_{n \in \mathbb{Z}}$ of projective objects that agrees with a projective resolution of $A$ in sufficiently high degree~\cite[Definition~1.1]{cor97}. Following~\cite[p.~158]{far77}, we define the Tate--Farrell Ext-functor $\overline{\mathrm{Ext}}_{\mathcal{C}}^{\bullet}(A, B)$ as the cohomology of the cochain complex $\mathrm{Hom}_{\mathcal{C}}(\overline{A}_{\bullet}, B)$ where the Hom-functor is unenriched.

\begin{lem} (= Lemma~\ref{lem:tatefarrell})
The Tate--Farrell Ext-functors $\overline{\mathrm{Ext}}_{\mathcal{C}}^{\bullet}(A, -)$ are isomorphic to the completed (unenriched) Ext-functors $\widehat{\mathrm{Ext}}_{\mathcal{C}}^{\bullet}(A, -)$ as cohomological functors.
\end{lem}

\subsection{Section summary}

We showcase in Section~\ref{sec:constructions} the relevant constructions of Mislin completions that we have generalised in~\cite{ghe24} for the reader's convenience. Several properties of completed Ext-functors are established in Section~\ref{sec:complextandcan}. More specifically, we show that zeroeth completed Ext-functors detect finite projective dimension, that completed Ext-functors generalise Tate--Farrell Ext-functors and that complete cohomology detects finiteness of pro-$p$ groups. In Section~\ref{sec:dimshifteckmannshapiro}, we prove a (partial) version of dimension shifting and an Eckmann--Shapiro type lemma. To pave the way for cohomology products, we provide an overview of external products of Ext-functors and cup products of group cohomology in Section~\ref{sec:ordinarycohomprods}. Then, in Section~\ref{sec:cohomprodsexists}, we construct external and Yoneda products for completed Ext-functors and cup products for completed group cohomology. We conclude by proving the most fundamental properties of these cohomology products in Section~\ref{sec:cohomprodpropties}.

\subsection{Notation and terminology}

We adopt the convention that the natural numbers $\mathbb{N}$ include $1$, but do not include $0$. We write $\mathbb{N}_0$ for $\mathbb{N} \cup \lbrace 0 \rbrace$. Moreover, we adopt B.\! Poonen's convention from~\cite{poo18} that a ring is an abelian group together with a totally associative product, meaning a binary associative relation admitting an identity element. In particular, ring homomorphisms are understood to map identity elements to identity elements. We use the symbol $\varinjlim$ only to denote direct limits. If $\mathcal{D}$ is a category, then $\varinjlim_{\mathcal{D}}$ denotes a direct limit in this category and if $I$ is a directed set, then $\varinjlim_{i \in I}$ denotes a direct limit indexed over $I$. We write $\varinjlim_{k \in \mathbb{N}} (M_k, \mu_k)$ if $\mu_k: M_k \rightarrow M_{k+1}$ are the morphisms giving rise to the direct limit. Lastly, we use the same numbering to label diagrams and equations.

\section{Outline of constructions}\label{sec:constructions}

This section provides an overview of the constructions of Mislin completion we have generalised in~\cite{ghe24}. More specifically, these constructions originate from D.\! J.\! Benson and J.\! F.\! Carlson's paper~\cite{ben92}, from F.\! Goichot's paper~\cite{goi92} and from G.\! Mislin's paper~\cite{mis94}. In order to present these, let us axiomatically define cohomological functors. For two abelian categories $\mathcal{C}$, $\mathcal{D}$ a family of additive functors $(T^n: \mathcal{C} \rightarrow \mathcal{D})_{n \in \mathbb{Z}}$ is called a cohomological functor if it satisfies the following two axioms~\cite[p.~201--202]{kro95}.

\begin{axm}\label{axm:delta}
For every $n \in \mathbb{Z}$ and short exact sequence $0 \rightarrow A \rightarrow B \rightarrow C \rightarrow 0$ in $\mathcal{C}$, there is natural connecting homomorphism $\delta^n: T^n(C) \rightarrow T^{n+1}(A)$.
\end{axm}

Being natural means in this context that for every commuting diagram in $\mathcal{C}$
\begin{center}
\begin{tikzcd}
0 \arrow[r] & A \arrow[r] \arrow[d, "f"] & B \arrow[r] \arrow[d] & C \arrow[r] \arrow[d, "g"] & 0 \\
0 \arrow[r] & A' \arrow[r] & B' \arrow[r] & C' \arrow[r] & 0
\end{tikzcd}
\end{center}
with exact rows there is a commuting diagram
\begin{center}
\begin{tikzcd}
T^n(A) \arrow[r, "\delta^n"] \arrow[d, "T^n(g)"] & T^{n+1}(A) \arrow[d, "T^{n+1}(f)"] \\
T^n(A') \arrow[r, "\delta^n"] & T^{n+1}(C')
\end{tikzcd}
\end{center}
in $\mathcal{D}$.

\begin{axm}\label{axm:les}
For every short exact sequence $0 \rightarrow A \xrightarrow{\iota} B \xrightarrow{\pi} C \rightarrow 0$ in $\mathcal{C}$ there is a long exact sequence
\[ \dots {} \xrightarrow{T^{n-1}(\pi)} T^{n-1}(C) \xrightarrow{\delta^{n-1}} T^n(A) \xrightarrow{T^n(\iota)} T^n(B) \xrightarrow{T^n(\pi)} T^n(C) \xrightarrow{\delta^n} T^{n+1}(A) \xrightarrow{T^{n+1}(\iota)} {} \dots \]
\end{axm}

Let us axiomatically define morphisms of cohomological functors as in~\cite[p.~202]{kro95}.

\begin{axm}\label{axm:morphisms}
Let $(T^{\bullet}, \delta^{\bullet})$ and $(U^{\bullet},\varepsilon^{\bullet})$ be cohomological functors from $\mathcal{C}$ to $\mathcal{D}$. Then a family of natural transformations $(\nu^n: T^n \rightarrow U^n)_{n \in \mathbb{Z}}$ is a morphism of cohomological functors if for every $n \in \mathbb{Z}$ and any short exact sequence $0 \rightarrow A \rightarrow B \rightarrow C \rightarrow 0$ in $\mathcal{C}$, the square
\begin{center}
\begin{tikzcd}
  T^n(C) \arrow[r, "\delta^n"] \arrow[d, "\nu^n"]
    &T^{n+1}(A) \arrow[d, "\nu^{n+1}"] \\
  U^n(C) \arrow[r, "\varepsilon^n"]
    &U^{n+1}(A)
\end{tikzcd}
\end{center}
commutes.
\end{axm}

We generalise G.\! Mislin's Definition~2.1 from~\cite{mis94} to the greatest extent.

\begin{defn}[Mislin completion]\label{defn:mislincompletion}
Let $(T^{\bullet}, \delta^{\bullet})$ be a cohomological functor from $\mathcal{C}$ to $\mathcal{D}$. Then its Mislin completion is a cohomological functor $(\widehat{T}^{\bullet}, \widehat{\delta}^{\bullet})$ from $\mathcal{C}$ to $\mathcal{D}$ together with a morphism $\nu^{\bullet}: T^{\bullet} \rightarrow \widehat{T}^{\bullet}$ satisfying the following universal property: \newline
\textnormal{1.} $(\widehat{T}^{\bullet}, \widehat{\delta}^{\bullet})$ vanishes on projectives, meaning that $\widehat{T}^n(P) = 0$ for every projective object $P \in \mathrm{obj}(\mathcal{C})$ and every $n \in \mathbb{Z}$. \newline
\textnormal{2.} If $(U^{\bullet}, \varepsilon^{\bullet})$ is any cohomological functor vanishing on projectives, then each morphism $T^{\bullet} \rightarrow U^{\bullet}$ factors uniquely as $T^{\bullet} \xrightarrow{\nu^{\bullet}} \widehat{T}^{\bullet} \rightarrow U^{\bullet}$.
\end{defn}

By virtue of their universal property, Mislin completions are unique up to isomorphism in the following sense. If $(U^{\bullet}, \varepsilon^{\bullet})$ is another Mislin completion of $(T^{\bullet}, \delta^{\bullet})$, then there is an isomorphism $\mu^{\bullet}: \widehat{T}^{\bullet} \rightarrow U^{\bullet}$, meaning that $\mu^n(M): \widehat{T}^n(M) \rightarrow U^n(M)$ is an isomorphism for any $n \in \mathbb{Z}$ and $M \in \mathrm{obj}(\mathcal{C})$~\cite[p.~202]{kro95}. This allows us to state G.\! Mislin's definition from~\cite[p.~297]{mis94} in the greatest generality.

\begin{defn}[Axiomatic, Mislin]\label{defn:axiomatic}
${} \quad {}$
\begin{itemize}
\item For any $A \in \mathrm{obj}(\mathcal{C})$ extend the (enriched or unenriched) Ext-functors to a cohomological functor by setting $\mathrm{Ext}_{\mathcal{C}}^n(A, -) = 0$ for $n < 0$. Define completed Ext-functors as the Mislin completion $(\widehat{\mathrm{Ext}}_{\mathcal{C}}^{\bullet}(A, -), \widehat{\delta}^{\bullet})$.
\item Analogously, if $G$ is a group object in $\mathcal{C}$ and $R$ a ring object, then extend group cohomology to a cohomological functor $(H_R^{\bullet}(G, -), \delta^{\bullet})$ by imposing $H_R^n(G, -) = 0$ for $n < 0$. Define complete cohomology as the Mislin completion $(\widehat{H}_R^{\bullet}(G, -), \widehat{\delta}^{\bullet})$.
\end{itemize}
\end{defn}

To ensure that Mislin completions exists, we explain what direct limits are and when they are called exact.

\begin{defn}\label{defn:directlimit}
A partially ordered set $(I, \leq)$ is a directed set if for every $i, j \in I$ there is $k \in I$ such that $i,j \leq k$~\cite[p.~1]{rib10}. According to~\cite[p.~14]{rib10}, a diagram $\lbrace D_i \rbrace_{i \in I}$ in $\mathcal{D}$ indexed over a directed set is called a direct system in $\mathcal{D}$. More formally, $I$ can be turned into a category whose objects are its elements and there is a unique morphism $i \rightarrow j$ whenever $i \leq j \in I$. Then a direct system is a covariant functor $I \rightarrow \mathcal{D}, i \mapsto D_i$. A direct limit $\varinjlim_{i \in I} D_i$ in $\mathcal{D}$ is a colimit of a direct system $\lbrace D_i \rbrace_{i \in I}$. Direct limits in $\mathcal{D}$ are called exact if for every direct system of short exact sequence $\lbrace 0 \rightarrow A_i \rightarrow B_i \rightarrow C_i \rightarrow 0 \rbrace_{i \in I}$ also
\[ 0 \rightarrow \varinjlim_{i \in I} A_i \rightarrow \varinjlim_{i \in I} B_i \rightarrow \varinjlim_{i \in I} C_i \rightarrow 0 \]
is a short exact sequence~\cite[\href{https://stacks.math.columbia.edu/tag/079A}{Tag 079A}]{stacks-project}.
\end{defn}

\begin{thm}
(\cite[Theorem~3.2]{ghe24}) Let $\mathcal{C}$ be an abelian category with enough projective objects and let $\mathcal{D}$ be an abelian category in which all direct limits exist and are exact. Then for every cohomological functor $(T^{\bullet}, \delta^{\bullet})$ from $\mathcal{C}$ to $\mathcal{D}$ there exists a Mislin completion $(\widehat{T}^{\bullet}, \widehat{\delta}^{\bullet})$.
\end{thm}

\begin{notn}
For the rest of the paper, $\mathcal{C}$ always denotes an abelian category with enough projectives and $\mathcal{D}$ an abelian category in which all countable direct limits exist and are exact.
\end{notn}

In order to clarify where the above assumptions are needed, we present what we term the satellite functor construction, which is due to G.\! Mislin. First, we introduce left satellite functors. Using the above assumption that $\mathcal{C}$ has enough projective objects, there is for any $M \in \mathrm{obj}(\mathcal{C})$ a short exact sequence $0 \rightarrow K \rightarrow P \rightarrow M \rightarrow 0$ in $\mathcal{C}$ with $P$ projective. For a cohomological functor $(T^{\bullet}, \delta^{\bullet})$, we define the zeroeth left satellite functor of $T^n$ as $S^0T^n := T^n$, the first left satellite functor as
\[ S^{-1}T^n(M) := \mathrm{Ker}\big(T^n(K) \rightarrow T^n(P)\big) \]
and the $k^{\text{th}}$ left satellite functor as $S^{-k}T^n := S^{-1}(S^{-k+1}T^n)$ for $k \geq 2$~\cite[p.~36]{car56}. It is shown in~\cite[Section~III.1]{car56} that left satellite functors do not depend on the choice of short exact sequence. Since they have been defined as kernels, it follows from Axiom~\ref{axm:les} that $\delta^n: T^n \rightarrow T^{n+1}$ induces a morphism $\underline{\delta}^n: T^n(M) \rightarrow S^{-1}T^{n+1}(M)$ and therefore
\[ S^{-k}\underline{\delta}^{n+k}: S^{-k}T^{n+k}(M) \rightarrow S^{-k-1}T^{n+k+1}(M) \]
for any $k \in \mathbb{N}$~\cite[p.~207--208]{kro95}. We extend G.\! Mislin's construction from~\cite[p.~293]{mis94} by explicitly using our assumption that all countable  direct limits exist in the codomain category $\mathcal{D}$ of $T^{\bullet}$.

\begin{defn}[Satellite functor construction, Mislin]
The Mislin completion of a cohomological functor $(T^{\bullet}, \delta^{\bullet})$ can be defined as
\[ \widehat{T}^n(M) := \varinjlim_{k \in \mathbb{N}_0} (S^{-k}T^{n+k}(M), S^{-k}\underline{\delta}^{n+k}) \]
for any $M \in \mathrm{obj}(\mathcal{C})$ and $n \in \mathbb{Z}$. Accordingly,
\[ \widehat{H}_R^n(G, M) := \varinjlim_{k \in \mathbb{N}_0} (S^{-k}H_R^{n+k}(G, M), S^{-k}\underline{\delta}^{n+k}) \]
is a definition of complete cohomology.
\end{defn}

In order that the above forms a cohomological functor, we require the very assumption that all direct limits in $\mathcal{D}$ are exact. \\

Let us go over to what we term the resolution construction that occurs in~\cite[Lemma~B.3]{cel17} and can be retrieved from page~299 in G.\! Mislin's paper~\cite{mis94}. If $(M_n)_{n \in \mathbb{N}_0}$ is a projective resolution of $M \in \mathrm{obj}(\mathcal{C})$, let us define $\widetilde{M}_0 := M$ and $\widetilde{M}_k := \mathrm{Ker}(M_{k-1} \rightarrow \widetilde{M}_{k-1})$ for $k \in \mathbb{N}$. This is called the $k^{\text{th}}$ syzygy of $M_{\bullet}$ in the Gorenstein context~\cite[p.~89]{pag16}. The choice of our notation is meant to reflect that our syzygies do not necessarily arise from a specific choice of projective resolution as in \cite{kro95} and \cite{mis94}. Since for every $k \in \mathbb{N}_0$ we have the short exact sequence $0 \rightarrow \widetilde{M}_{k+1} \rightarrow M_k \rightarrow \widetilde{M}_k \rightarrow 0$, there is a connecting homomorphism $\delta^{n+k}: T^{n+k}(\widetilde{M}_k) \rightarrow T^{n+k+1}(\widetilde{M}_{k+1})$ for every $n \in \mathbb{Z}$. Then the following definition makes it more apparent why Mislin completions vanish on projective objects.

\begin{defn}[Resolution construction, Mislin]\label{defn:resolsinoutline}
The Mislin completion of a cohomological functor $(T^{\bullet}, \delta^{\bullet})$ can be defined as 
\[ \widehat{T}^n(M) := \varinjlim_{k \in \mathbb{N}_0} (T^{n+k}(\widetilde{M}_k), \delta^{n+k}) \]
for any $n \in \mathbb{Z}$ and $M \in \mathrm{obj}(\mathcal{C})$. Accordingly,
\[ \widehat{H}_R^n(G, M) := \varinjlim_{k \in \mathbb{N}_0} (H_R^{n+k}(G, \widetilde{M}_k), \delta^{n+k}) \]
is a definition of complete cohomology.
\end{defn}

The next two constructions only give rise to completed unenriched Ext-functors. Let $A_{\bullet}$, $B_{\bullet}$ are projective resolutions of $A, B \in \mathrm{obj}(\mathcal{C})$ and let $\widetilde{f}_{n+k}: \widetilde{A}_{n+k} \rightarrow \widetilde{B}_k$ be a morphism for $n \in \mathbb{Z}$ and $k \in \mathbb{N}_0$ such that $n+k \geq 0$. Then we can write the commuting diagram
\begin{center}
\begin{tikzcd}
    0 \arrow[r] & \widetilde{A}_{n+k+1} \arrow[r] \arrow[d, "\widetilde{f}_{k+1}"] & A_{n+k} \arrow[r] \arrow[d, "f_{k+1}"] & \widetilde{A}_{n+k} \arrow[r] \arrow[d, "\widetilde{f}_k"] & 0 \\
    0 \arrow[r] & \widetilde{B}_{k+1} \arrow[r] & B_{k+1} \arrow[r] & \widetilde{B}_k \arrow[r] & 0
\end{tikzcd}
\end{center}
whose terms arise as follows. Since the bottom row is exact and the term $A_{n+k}$ projective, there is a lift $f_k$ of $\widetilde{f}_k$ making the right-hand square commute. Because $\widetilde{B}_{k+1} \rightarrow B_{k+1}$ is a kernel, there is a morphism $\widetilde{f}_{k+1}$ making the left-hand side commute. If $\mathrm{Hom}_{\mathcal{C}}(-, -)$ denotes the (unenriched) Hom-functor in $\mathcal{C}$, then $\mathrm{Hom}_{\mathcal{C}}(\widetilde{A}_{n+k}, \widetilde{B}_k)$ is an abelian group by virtue of $\mathcal{C}$ being an abelian category. We define $\mathcal{P}_{\mathcal{C}}(\widetilde{A}_{n+k}, \widetilde{B}_k)$ to be the subgroup of $\mathrm{Hom}_{\mathcal{C}}(\widetilde{A}_{n+k}, \widetilde{B}_k)$ consisting of all morphisms factoring through a projective object and write the quotient as $[\widetilde{A}_{n+k}, \widetilde{B}_k]_{\mathcal{C}} := \mathrm{Hom}_{\mathcal{C}}(\widetilde{A}_{n+k}, \widetilde{B}_k)/ \mathcal{P}_{\mathcal{C}}(\widetilde{A}_{n+k}, \widetilde{B}_k)$~\cite[p.~203]{kro95}. As in the case of modules over a ring covered by~\cite[p.~204]{kro95}, one can prove that
\begin{align*}
t_{\widetilde{A}_{n+k}, \widetilde{B}_k}: [\widetilde{A}_{n+k}, \widetilde{B}_k]_{\mathcal{C}} &\rightarrow [\widetilde{A}_{n+k+1}, \widetilde{B}_{k+1}]_{\mathcal{C}}, \\
\widetilde{f}_k + \mathcal{P}_{\mathcal{C}}(\widetilde{A}_{n+k}, \widetilde{B}_k) &\mapsto \widetilde{f}_{k+1} + \mathcal{P}_{\mathcal{C}}(\widetilde{A}_{n+k+1}, \widetilde{B}_{k+1})
\end{align*}
is a well defined homomorphism. Using this, we generalise the following construction from D.\! J.\! Benson and J.\! F.\! Carlson's paper~\cite[p.~109]{ben92}.

\begin{defn}[Naïve construction, Benson \& Carlson]
For any $n \in \mathbb{Z}$, we can define the $n^{\text{th}}$ completed unenriched Ext-functor as
\[ \widehat{\mathrm{Ext}}_{\mathcal{C}}^n(A, B) := \varinjlim_{k \in \mathbb{N}_0, n+k \geq 0} ([\widetilde{A}_{n+k}, \widetilde{B}_k]_{\mathcal{C}}, t_{\widetilde{A}_{n+k}, \widetilde{B}_k}) \, . \]
In particular, if $R_{\bullet}$ is a projective $R[G]$-resolution of $R \in \mathrm{obj}(Mod_R(G))$, we can define complete unenriched cohomology as
\[ \widehat{H}_R^n(G, B) := \varinjlim_{k \in \mathbb{N}_0, n+k \geq 0} ([\widetilde{R}_{n+k}, \widetilde{B}_k]_{\mathcal{C}}, t_{\widetilde{R}_{n+k}, \widetilde{B}_k}) \, . \]
\end{defn}

Lastly, we present what we call the hypercohomology construction of complete cohomology. We define the chain complex $(A_n')_{n \in \mathbb{Z}}$ by $A_n' = A_n$ for $n \geq 0$ and $A_n' = 0$ for $n < 0$ and similarly $(B_n')_{n \in \mathbb{Z}}$~\cite[p.~209]{kro95}. Define the hypercohomology complex $(\mathrm{Hyp}_{\mathcal{C}}(A_{\bullet}', B_{\bullet}')_n, d^n)_{n \in \mathbb{Z}}$ by having $n$-cochains 
\[ \mathrm{Hyp}_{\mathcal{C}}(A_{\bullet}', B_{\bullet}')_n = \prod_{k \in \mathbb{Z}} \mathrm{Hom}_{\mathcal{C}}(A_{k+n}', B_k') \, . \]
To ease notation in the following, we view abelian groups as $\mathbb{Z}$-modules. If we denote by $a_n: A_n' \rightarrow A_{n-1}'$ and $b_n: B_n' \rightarrow B_{n-1}'$ the differentials induced from the respective projective resolution, we define for $n \in \mathbb{Z}$ the differential
\begin{align}
d^n: \mathrm{Hyp}_{\mathcal{C}}(A_{\bullet}', B_{\bullet}')_n &\rightarrow \mathrm{Hyp}_{\mathcal{C}}(A_{\bullet}', B_{\bullet}')_{n+1} \label{eq:hypercohombd} \\ (\varphi_{n+k})_{k \in \mathbb{Z}} &\mapsto (b_{k+1} \circ \varphi_{n+k+1} - (-1)^n \varphi_{n+k} \circ a_{n+k+1})_{k \in \mathbb{Z}} \nonumber
\end{align}
Let us define the bounded complex $\mathrm{Bdd}_{\mathcal{C}}(A_{\bullet}', B_{\bullet}')_{n \in \mathbb{Z}}$ as the subcomplex of $\mathrm{Hyp}_{\mathcal{C}}(A_{\bullet}', B_{\bullet}')_{n \in \mathbb{Z}}$ given by
\[ \mathrm{Bdd}_{\mathcal{C}}(A_{\bullet}', B_{\bullet}')_n = \bigoplus_{k \in \mathbb{Z}} \mathrm{Hom}_{\mathcal{C}}(A_{k+n}', B_k') \, . \]
Define the Vogel complex as the quotient complex~\cite[p.~209]{kro95}
\[ \mathrm{Vog}_{\mathcal{C}}(A_{\bullet}', B_{\bullet}')_{n \in \mathbb{Z}} := \Big(\mathrm{Hyp}_{\mathcal{C}}(A_{\bullet}', B_{\bullet}')_n / \mathrm{Bdd}_{\mathcal{C}}(A_{\bullet}', B_{\bullet}')_n \Big)_{n \in \mathbb{Z}} \, . \]
By this, we generalise Definition~1.2 from F.\! Goichot's paper~\cite{goi92} where he attributes it to P.\! Vogel on page~39.

\begin{defn}[Hypercohomology construction, Vogel]\label{defn:vogel}
For $n \in \mathbb{Z}$ we can define the $n^{\text{th}}$ completed unenriched Ext-functor as
\[ \widehat{\mathrm{Ext}}_{\mathcal{C}}^n(A, B) := H^n((\mathrm{Vog}_{\mathcal{C}}(A_{\bullet}', B_{\bullet}')_k, d^k)_{k \in \mathbb{Z}}) \, . \]
We can thus define complete unenriched cohomology as
\[ \widehat{H}_R^n(G, M) := H^n((\mathrm{Vog}_R(R_{\bullet}', B_{\bullet}')_k, d^k)_{k \in \mathbb{Z}})\, . \]
\end{defn}

\section{Completed Ext-functors and canonical morphisms}\label{sec:complextandcan}

This section is dedicated to properties of completed Ext-functors and their canonical morphisms. We show that zeroeth completed unenriched Ext-functors detect finite projective dimension. We show that completed unenriched Ext-functors generalise Tate--Farrell Ext-functors. We demonstrate that complete unenriched group cohomology detects when a pro-$p$ group is finite. Lastly, we prove that the terms of the canonical morphism from unerniched Ext-functors to their Mislin completion fit into a long exact sequence relating three distinct cohomological functors. \\

Recall from Section~\ref{subsec:findim} that an object in a category is said to have finite projective dimension if it admit a projective resolution of finite length. Remember that group cohomology as well as complete cohomology can be defined as a specific (completed) Ext-functor. Then we re-establish Lemma~4.2.3 from~\cite{kro95} in greater generality where we refer the reader to~\cite[Section~7]{ghe24} for an implementation of complete cohomology into condensed mathematics.

\begin{lem}\label{lem:prevanishing}
\begin{enumerate}
    \item Assume that $T^{\bullet}: \mathcal{C} \rightarrow \mathcal{D}$ is a cohomological functor where $\mathcal{C}$ has enough projectives and in $\mathcal{D}$ all countable direct limits exist and are exact. If $M \in \mathrm{obj}(\mathcal{C})$ has finite projective dimension, then $\widehat{T}^n(M) = 0$ for every $n \in \mathbb{Z}$. In particular, if every object in $\mathcal{C}$ has finite projective dimension such as in a category of modules over a ring of finite global dimension, then $\widehat{T}^{\bullet} = 0$ for any cohomological functor $T^{\bullet}$.
    \item If one takes \textbf{enriched} Ext-functors $\mathrm{Ext}_{\mathcal{C}}^n(A, -): \mathcal{C} \rightarrow \mathcal{D}$ with $A \in \mathrm{obj}(\mathcal{C})$ of finite projective dimension, then $\widehat{\mathrm{Ext}}_{\mathcal{C}}^n(A, -) = 0$ for every $n \in \mathbb{Z}$. In particular, complete cohomology $\widehat{H}_R^{\bullet}(G, M)$ vanishes if the group object $G$ has finite cohomological dimension or $M \in Mod_R(G)$ has finite projective dimension. As this applies to any condensed group $G$, this holds for any $T1$ topological group.
\end{enumerate}
\end{lem}

\begin{proof}
If $M$ has finite projective dimension, then there is $m \in \mathbb{N}_0$ such that $\widetilde{M}_k = 0$ for any $k \geq m$. In particular, $T^{n+k}(\widetilde{M}_k) = 0$ for any $k \geq m$ and $\widehat{T}^n(M) = 0$ according to the resolution construction. On the other hand, if $A$ has finite projective dimension, then there $m' \in \mathbb{N}_0$ such that $\mathrm{Ext}_{\mathcal{C}}^{n+k}(A, -) = 0$ for any $n+k \geq m'$. We conclude as before that $\widehat{\mathrm{Ext}}_{\mathcal{C}}^n(A, -) = 0$.
\end{proof}

This allows us to re-establish the following version of a theorem that appears in the literature as~\cite[Proposition~IX.1.3]{bel07}, \cite[Theorem~4.11]{guo23}, \cite[Theorem~3.10]{hu21} and~\cite[Lemma~4.2.4]{kro95}.

\begin{thm}\label{thm:realvanishing}
If $\widehat{\mathrm{Ext}}_{\mathcal{C}}^{\bullet}(A, -): \mathcal{C} \rightarrow \mathbf{Ab}$ denote completed unenriched Ext-functors for $A \in \mathrm{obj}(\mathcal{C})$, then the following are equivalent.
\begin{enumerate}
\item The object $A$ has finite projective dimension.
\item $\widehat{\mathrm{Ext}}_{\mathcal{C}}^n(A, -) = \widehat{\mathrm{Ext}}_{\mathcal{C}}^n(-, A) = 0$ for any $n \in \mathbb{Z}$.
\item $\widehat{\mathrm{Ext}}_{\mathcal{C}}^0(A, A) = 0$.
\end{enumerate}
In particular, the zeroeth complete cohomology group detects whether a group has finite cohomological dimension. This applies to any condensed group and thus to any $T1$ topological group.
\end{thm}

Given the torsion-theoretic framework in which A.\! Beligiannis and I.\! Reiten work, this theorem follows from their definitions in~\cite{bel07}. On the other hand, S.\! Guo and L.\! Liang in~\cite{guo23} as well as J.\! Hu et al.\! in~\cite{hu21} prove this theorem using the hypercohomology construction. For completeness, we generalise a proof by the naïve construction found in~\cite[p.~205]{kro95}.

\begin{proof}
As the second statement implies the third, it suffices by Lemma~\ref{lem:prevanishing} to prove that $A$ has finite projective dimension if $\widehat{\mathrm{Ext}}_{\mathcal{C}}^0(A, A) = 0$. By the construction of direct limits of abelian groups from~\cite[p.~261]{osb00}, there is $k \in \mathbb{N}_0$ such that $\mathrm{id}_{\widetilde{A}_k} + \mathcal{P}_{\mathcal{C}}(\widetilde{A}_k, \widetilde{A}_k) = 0$ in $[\widetilde{A}_k, \widetilde{A}_k]_{\mathcal{C}}$. Because unenriched Hom-functors are used, we conclude that $\mathrm{id}_{\widetilde{A}_k}$ factors through a projective. In particular, $\widetilde{A}_k$ is projective and $A$ has finite projective dimension.
\end{proof}

To prove that completed Ext-functors generalise Tate--Farrell Ext-functors, we require a result on when the terms of a cohomological functors agree with the ones of its Mislin completion. The following is a generalisation of Lemma~2.3 in~\cite{mis94}.

\begin{prop}\label{prop:mislincohomgroups}
Let $T^{\bullet}: \mathcal{C} \rightarrow \mathcal{D}$ be a cohomological functor. Then for $n \in \mathbb{Z}$ the following are equivalent.
\begin{enumerate}
\item For any $k \geq n$ the functor $T^k$ vanishes on projective objects.
\item For any $k \geq n$ the functor $T^k$ is naturally isomorphic to $\widehat{T}^k$.
\end{enumerate}
\end{prop}

\begin{proof}
As the second assertion implies the first, assume that $T^k$ vanishes on projective objects for any $k \geq n$. Let $T^{\bullet} \rightarrow U^{\bullet}$ be a morphism of cohomological functors where $U^{\bullet}$ vanishes on projective objects. By Lemma~3.1 and the subsequent satellite functor construction found in~\cite{ghe24}, this morphism factors uniquely as $T^{\bullet} \rightarrow T^{\bullet}\langle n \rangle \rightarrow U^{\bullet}$. Because $T^{\bullet}\langle n \rangle$ vanishes on projectives by our assumptions, we infer that it is a Mislin completion. In particular, we observe for any $k \geq n$ that $T^k = T^k \langle n \rangle \cong \widehat{T}^k$.
\end{proof}

We also require a criterion for when a cohomological functor vanishing on projectives is a Mislin completion. The following is a generalisation of Lemma~2.4 in~\cite{mis94}.

\begin{prop}\label{prop:detamislincompletion}
Let $T^{\bullet}, V^{\bullet}: \mathcal{C} \rightarrow \mathcal{D}$ be cohomological functors where $V^{\bullet}$ vanishes on projective objects. Assume that $\Phi^{\bullet}: T^{\bullet} \rightarrow V^{\bullet}$ is a morphism such that there is $n \in \mathbb{Z}$ with the property that $\Phi^k$ is an isomorphism for any $k \geq n$. Then $V^{\bullet}$ together with $\Phi^{\bullet}$ forms a Mislin completion of $T^{\bullet}$.
\end{prop}

\begin{proof}
Note that $T^k$ vanishes on projectives for any $k \geq n$. It follows from the proof of Proposition~\ref{prop:mislincohomgroups} that there is a unique factorisation $\Phi^{\bullet}: T^{\bullet} \rightarrow T^{\bullet}\langle n \rangle \xrightarrow{\Psi^{\bullet}} V^{\bullet}$ where $T^{\bullet}\langle n \rangle$ is the Mislin completion of $T^{\bullet}$. In particular, $\Psi^k = \Phi^k$ is an isomorphism for any $k \geq n$. Therefore, $\Psi^{\bullet}$ is an isomorphism of cohomological functors by Lemma~3.1 in~\cite{ghe24} and $V^{\bullet}$ is a Mislin completion.
\end{proof}

To define Tate--Farrell Ext-functors, a complete resolution of $A \in \mathrm{obj}(\mathcal{C})$ is an acyclic chain complex $(\overline{A}_n)_{n \in \mathbb{Z}}$ of projectives satisfying the following two properties~\cite[Definition~1.1]{cor97}.
\begin{itemize}
\item There is $n \in \mathbb{N}_0$ for which $(\overline{A}_k)_{k \geq n}$ agrees with a projective resolution of $A$.
\item If $\mathrm{Hom}_{\mathcal{C}}(-, -)$ denotes the unenriched Hom-functor, then the cochain complex $\mathrm{Hom}_{\mathcal{C}}(\overline{A}_{\bullet}, P)$ is acyclic for any projective $P \in \mathrm{obj}(\mathcal{C})$.
\end{itemize}
In accordance with~\cite[p.~158]{far77}, we define $\overline{\mathrm{Ext}}_{\mathcal{C}}^{\bullet}(A, B) := H^{\bullet}(\mathrm{Hom}_{\mathcal{C}}(\overline{A}_{\bullet}, B))$. By~\cite[Lemma~2.4]{cor97}, any two complete resolutions are chain homotopic and thus, Tate--Farrell Ext-functors do not depend on the choice of a complete resolution. In~\cite[Proposition~4.15]{guo23}, S.\! Guo and L.\! Liang demonstrate that the Tate--Farrell Ext-functor $\overline{\mathrm{Ext}}_{\mathcal{C}}^n(A, -)$ is naturally isomorphic to $\widehat{\mathrm{Ext}}_{\mathcal{C}}^n(A, -)$ for every $n \in \mathbb{Z}$. We strengthen their result to an isomorphism of cohomological functors and thus obtain the following version of Theorem~4.6 from~\cite{hu21}.

\begin{lem}\label{lem:tatefarrell}
The Tate--Farrell Ext-functors $\overline{\mathrm{Ext}}_{\mathcal{C}}^{\bullet}(A, -)$ are isomorphic to the completed unenriched Ext-functors $\widehat{\mathrm{Ext}}_{\mathcal{C}}^{\bullet}(A, -)$ as cohomological functors.
\end{lem}

Apart from the proof found in~\cite{hu21}, one can demonstrate this by using the proof of Theorem~1.2 in~\cite{cor97} together with Proposition~\ref{prop:detamislincompletion}. As indicated at the very start of Subsection~\ref{subsec:expo}, complete cohomology generalises Tate--Farrell cohomology for any group of finite virtual cohomological dimension. For the following example, we recall from Subsection~\ref{subsec:outline} that profinite groups are defined as Galois groups or equivalently, as inverse limits of finite groups.

\begin{exl}\label{exl:tatefarrell}
P.\! Symonds constructs in~\cite[p.~34]{sym07} Tate--Farrell cohomology for a profinite group with an open subgroup of finite cohomological dimension taking coefficients in profinite modules. Thus, complete cohomology generalises his Tate--Farrell cohomology.
\end{exl}

We generalise another result from discrete groups to profinite groups by proving that complete cohomology detects finite groups among pro-$p$ groups. For a fixed prime number $p$, we recall from Subsection~\ref{subsec:varia} that pro-$p$ groups are inverse limits of finite $p$-groups and are thus a profinite version of $p$-groups. As we are investigating cohomology of profinite groups, recall the completed group ring $R{\llbracket}G{\rrbracket}$ of a profinite group $G$ over a profinite ring $R$ from Subsection~\ref{subsec:shapiroshifting}. We require the following description of $H_R^0(G, M)$.

\begin{prop}\label{prop:zeroethgaloiscohom}
(\cite[Lemma~6.2.1]{rib10}) Let $G$ be a profinite group and $R$ a profinite commutative ring. For any discrete or profinite $R{\llbracket}G{\rrbracket}$-module $M$ define
\[ M^G := \lbrace m \in M \mid \forall g \in G: g \cdot m = m \rbrace \, . \]
Then $H_R^0(G, M) \cong \mathrm{Hom}_{R{\llbracket}G{\rrbracket}}(R, M) \cong M^G$.
\end{prop}

We restrict our attention to the $p$-adic integers $R = \mathbb{Z}_p$ which are an important example of a pro-$p$ ring~\cite[Section~1.5]{wil98}. In the case of the $\mathbb{Z}_p{\llbracket}G{\rrbracket}$-module $\mathbb{Z}_p[G/U]$ with $U \trianglelefteq G$ open, we provide a more explicit description of $H_{\mathbb{Z}_p}^0(G, \mathbb{Z}_p[G/U])$.

\begin{prop}\label{prop:specialzeroethcohom}
For any profinite group $G$ and any open subgroup $U \trianglelefteq G$ we have
\[ \mathbb{Z}_p[G/U]^G = \Big\lbrace (b_y)_{y \in G/U} \in \prod_{y \in G/U} \mathbb{Z}_p \mid \exists b \in \mathbb{Z}_p \, \forall y \in G/U: b_y = b \Big\rbrace \, . \]
If $V \trianglelefteq U \trianglelefteq  G$ is another open subgroup and $G/V \rightarrow G/U$ denotes the projection homomorphism, then the induced homomorphism is given by
\[ \mathbb{Z}_p[G/V] \rightarrow \mathbb{Z}_p[G/U], (b)_{y \in G/V} \rightarrow (|U:V| \cdot b)_{y' \in G/U} \, . \]
\end{prop}

\begin{proof}
Let $(b_y)_{y \in G/U} \in \mathbb{Z}_p[G/U]^G$. Then $(b_{g \cdot y})_{y \in G/U} = g \cdot (b_y)_{y \in G/U} = (b_y)_{y \in G/U}$ for any $g \in G$. According to~\cite[Section~0.9]{dix99} the $p$-adic integers $\mathbb{Z}_p$ can be embedded into the $p$-adic rationals $\mathbb{Q}_p$ which are a topological field. In particular, one can embed $\mathbb{Z}_p[G/U]$ into the $\mathbb{Q}_p$-vector space $\mathbb{Q}_p^{|G:U|}$. This implies that $b_y = b_{g \cdot y}$ for any $y \in G/U$ and $g \in G$, proving the first assertion. For the second assertion, if $f: X \rightarrow Y$ is a map between two finite discrete spaces, then the corresponding induced homomorphism is given by
\[ \mathbb{Z}_p[f]: \mathbb{Z}_p[X] \rightarrow \mathbb{Z}_p[Y], \: (c_x)_{x \in X} \mapsto \Big(\sum_{x \in f^{-1}(y)} c_x \Big)_{y \in Y} \, . \]
Note now that the kernel of the group homomorphism $G/V \rightarrow G/U$ is given by $U/V$.
\end{proof}

This allows us to generalise Proposition~3.9 of~\cite{jo04} from discrete groups to pro-$p$ groups.

\begin{lem}\label{lem:propgroupfinite}
A pro-$p$ group $G$ is finite if and only if
\begin{itemize}
\item $H_{\mathbb{Z}_p}^0(G, -) \ncong \widehat{H}_{\mathbb{Z}_p}^0(G, -)$ and
\item $H_{\mathbb{Z}_p}^n(G, -) \cong \widehat{H}_{\mathbb{Z}_p}^n(G, -)$ for any $n \geq 1$.
\end{itemize}
\end{lem}

\begin{proof}
The forward implication follows from~\cite[p.~78--79]{ade04} because complete cohomology of a finite group $G$ agrees with Tate cohomology by~\cite{mis94}. Hence, assume that $H_{\mathbb{Z}_p}^0(G, -)$ is isomorphic to $\widehat{H}_{\mathbb{Z}_p}^0(G, -)$ in any degree greater, but not equal to $0$. By Proposition~\ref{prop:mislincohomgroups}, there is a projective $\mathbb{Z}_p{\llbracket}G{\rrbracket}$-module $P$ such that $H_{\mathbb{Z}_p}^0(G, P) \neq 0$. According to~\cite[Lemma~5.2.5]{rib10}, $P$ is a quotient of a free profinite module $(\mathbb{Z}_p{\llbracket}G{\rrbracket}){\llbracket}X{\rrbracket}$ on a profinite space $X$ where the quotient homomorphism has a (continuous) section. In particular, $H_{\mathbb{Z}_p}^0(G, (\mathbb{Z}_p{\llbracket}G{\rrbracket}){\llbracket}X{\rrbracket}) \neq 0$. By~\cite[Proposition~5.2.2]{rib10},
\[ (\mathbb{Z}_p{\llbracket}G{\rrbracket}){\llbracket}X{\rrbracket} = \varprojlim_{i \in I} (\mathbb{Z}_p{\llbracket}G{\rrbracket})[X_i] \]
where $X = \varprojlim_{i \in I} X_i$ with every $X_i$ a finite discrete space. Denote the corresponding $\mathbb{Z}_p{\llbracket}G{\rrbracket}$-homomorphisms by 
\[ f_{i, j}: (\mathbb{Z}_p{\llbracket}G{\rrbracket})[X_j] \rightarrow (\mathbb{Z}_p{\llbracket}G{\rrbracket})[X_i] \, . \]
By Proposition~\ref{prop:zeroethgaloiscohom}, $H_{\mathbb{Z}_p}^0(G, -) = \mathrm{Hom}_{\mathbb{Z}_p{\llbracket}G{\rrbracket}}(\mathbb{Z}_p, -)$ where such a Hom-functors preserve limits according to~\cite[p.~116]{mac78}. In particular, $\varprojlim_{i \in I} H_{\mathbb{Z}_p}^0(G, (\mathbb{Z}_p{\llbracket}G{\rrbracket})[X_i]) \neq 0$ as an inverse limit of abelian groups. If we use~\cite[p.~2--3]{rib10} and Proposition~\ref{prop:zeroethgaloiscohom}, we can describe it as the submodule of the product $\prod_{i \in I} H_{\mathbb{Z}_p}^0(G, (\mathbb{Z}_p{\llbracket}G{\rrbracket})[X_i])$ given by
\[ \Big\lbrace (a_i)_{i \in I} \in \prod_{i \in I} H_{\mathbb{Z}_p}^0(G, (\mathbb{Z}_p{\llbracket}G{\rrbracket})[X_i]) \mid \forall i \leq j: f_{i, j}(a_j) = a_i \Big\rbrace \, . \]
Thus, there is $i \in I$ such that $0 \neq a_i \in H_{\mathbb{Z}_p}^0(G, (\mathbb{Z}_p{\llbracket}G{\rrbracket})[X_i])$. As is noted in~\cite[p.~167]{rib10}, $(\mathbb{Z}_p{\llbracket}G{\rrbracket})[X_i] = \prod_{x \in X_i} \mathbb{Z}_p{\llbracket}G{\rrbracket}$ which implies that $H_{\mathbb{Z}_p}^0(G, \mathbb{Z}_p{\llbracket}G{\rrbracket}) \neq 0$. \\

Assume by contradiction that $G$ is infinite. Because the completed group ring can be given as $\mathbb{Z}_p{\llbracket}G{\rrbracket} = \varprojlim_{U \trianglelefteq G \text{ open}} \mathbb{Z}_p[G/U]$ by~\cite[p.~171]{rib10}, we may conclude the existence of a nonzero element $b_{U_0} \in H_{\mathbb{Z}_p}^0(G, \mathbb{Z}_p[G/U_0])$ as before. Moreover, because $G$ is infinite, we can also conclude that there is a countable sequence of open subgroups $U_{n+1} \triangleleftneq U_n \trianglelefteq G$ and of elements $b_{U_n} \in H_{\mathbb{Z}_p}^0(G, \mathbb{Z}_p[G/U_n])$ such that for any $m \leq n \in \mathbb{N}_0$ the projection homomorphism $\mathbb{Z}_p[G/U_n] \rightarrow \mathbb{Z}_p[G/U_m]$ maps $b_{U_n}$ to $b_{U_m}$. It follows from Proposition~\ref{prop:specialzeroethcohom} that there is a $b \in \mathbb{Z}_p$ such that $b_{U_0} = (b)_{x \in G/U_0} \in \prod_{x \in G/U_0} \mathbb{Z}_p$. Since $G$ is assumed to be a pro-$p$ group, it also follows that $b$ is divisible by arbitrarily large powers of $p$. However, this is impossible by~\cite[p.~26--27]{wil98} and thus $G$ is finite.
\end{proof}

Denote now by $\Phi^n: \mathrm{Ext}_{\mathcal{C}}^n(A, -) \rightarrow \widehat{\mathrm{Ext}}_{\mathcal{C}}^n(A, -)$ the terms of the canonical morphism to the Mislin completion as in Definition~\ref{defn:mislincompletion}. In order to prove that the terms $\Phi^n$ fit into a long exact sequence relating three distinct cohomological functors, we generalise a remark from~\cite[p.~210]{kro95} to

\begin{prop}\label{prop:quotinducescan}
For unenriched Ext-functors $\mathrm{Ext}_{\mathcal{C}}^{\bullet}(A, -): \mathcal{C} \rightarrow \mathbf{Ab}$ the quotient map $\mathrm{Hyp}_{\mathcal{C}}(A_{\bullet}, B_{\bullet})_{\bullet} \rightarrow \mathrm{Vog}_{\mathcal{C}}(A_{\bullet}, B_{\bullet})_{\bullet}$ of chain complexes from the hypercohomology construction induces the canonical morphism $\Phi^{\bullet}: \mathrm{Ext}_{\mathcal{C}}^{\bullet}(A, B) \rightarrow \widehat{\mathrm{Ext}}_{\mathcal{C}}^{\bullet}(A, B)$ of cohomological functors from the definition of a Mislin completion.
\end{prop}

\begin{proof}
Denote by $\Phi^{\bullet}: \mathrm{Ext}_{\mathcal{C}}^{\bullet}(A, -) \rightarrow \widehat{\mathrm{Ext}}_{\mathcal{C}}^{\bullet}(A, -)$ the canonical morphism of cohomological functors established through the satellite functor construction in~\cite[Theorem~3.2]{ghe24}. If $\omega_{\bullet}: \widehat{\mathrm{Ext}}_{\mathcal{C}}^{\bullet}(A, -) \rightarrow \mathrm{Ext}_{Res\mathcal{C}}^{\bullet}(A, -)$ is the isomorphism of cohomological functors to the resolution construction from~\cite[Theorem~4.6]{ghe24}, then
\[ \omega_{\bullet} \circ \Phi^{\bullet}: \mathrm{Ext}_{\mathcal{C}}^{\bullet}(A, -) \rightarrow \mathrm{Ext}_{Res, \mathcal{C}}^{\bullet}(A, -) \]
is also a canonical morphism to the Mislin completion. According to Diagram~4.5 of the proof of Lemma~4.1 in~\cite{ghe24}, each term $\omega_n \circ \Phi^n: \mathrm{Ext}_{\mathcal{C}}^n(A, B) \rightarrow \mathrm{Ext}_{Res, \mathcal{C}}^n(A, B)$ is a homomorphism to the direct limit occurring in the resolution construction as in Definition~\ref{defn:resolsinoutline}. If both $\mathcal{E}xt_{\mathcal{C}}^{\bullet}(A, -)$ and $\mathcal{E}xt_{Res, \mathcal{C}}^{\bullet}(A, -)$ are taken as in~\cite[Definition~6.9]{ghe24}, we denote by $\Psi^{\bullet}: \mathcal{E}xt_{\mathcal{C}}^{\bullet}(A, -) \rightarrow \mathcal{E}xt_{Res, \mathcal{C}}^{\bullet}(A, -)$ an analogous morphism to the direct limit. Taking the isomorphisms of cohomological functors $\zeta_{\bullet}$ and $\zeta^{\bullet}$ also from~\cite[Definition~6.9]{ghe24}, we see that the diagram
\begin{center}
\begin{tikzcd}
    \mathrm{Ext}_{\mathcal{C}}^{\bullet}(A{,} \, -) \arrow[r, "\zeta_{\bullet}"] \arrow[d, "\omega_{\bullet} \circ \Phi^{\bullet}"] & \mathcal{E}xt_{\mathcal{C}}^{\bullet}(A{,} \, -) \arrow[d, "\Psi^{\bullet}"] \\
    \mathrm{Ext}_{\mathcal{C}}^{Res, \bullet}(A{,} \, -) \arrow[r, "\zeta^{\bullet}"] & \mathcal{E}xt_{Res, \mathcal{C}}^{\bullet}(A{,} \, -)
\end{tikzcd}
\end{center}
commutes. Since $\mathcal{E}xt_{\mathcal{C}}^{\bullet}(A, -)$ is a different description of the Ext-functors $\mathrm{Ext}_{\mathcal{C}}^{\bullet}(A, -)$ according to~\cite[Notation~6.1]{ghe24}, $\Psi^{\bullet}$ also represents a canonical morphism to the Mislin completion. If $\vartheta_{\bullet}^0$ denote homomorphisms from the proof of~\cite[Lemma~6.12]{ghe24} and $\vartheta^{\bullet}$ the isomorphism of a cohomological functors from~\cite[Lemma~6.16]{ghe24}, then the diagram
\begin{center}
\begin{tikzcd}
    \mathcal{E}xt_{\mathcal{C}}^{\bullet}(A{,} \, -) \arrow[d, "\Psi^{\bullet}"] \arrow[dr, "\vartheta_{\bullet}^0"] & \\
    \mathcal{E}xt_{Res, \mathcal{C}}^{\bullet}(A{,} \, -) \arrow[r, "\vartheta^{\bullet}"] & \widehat{\mathcal{E}xt}_{\mathcal{C}}^{\bullet}(A{,} \, -)
\end{tikzcd}
\end{center}
is commutative. As before, we infer that $\vartheta_{\bullet}^0$ is a canonical morphism to the Mislin completion $\widehat{\mathcal{E}xt}_{\mathcal{C}}^{\bullet}(A, -)$. \\

Note that we can restrict the quotient map of chain complexes
\[ (\mathrm{Hyp}_{\mathcal{C}}(A_{\bullet}, B_{\bullet})_n, d^n)_{n \in \mathbb{Z}} \rightarrow (\mathrm{Vog}_{\mathcal{C}}(A_{\bullet}, B_{\bullet})_n, \overline{d}^n)_{n \in \mathbb{Z}} \]
to a homomorphism $\mathrm{Ker}(d^n) \rightarrow \mathrm{Ker}(\overline{d}^n)$ for any $n \in \mathbb{Z}$. Again by~\cite[Notation~6.1]{ghe24}, the latter is equivalent to $\mathrm{Hom}_{\mathrm{Ch}(\mathcal{C})}(A[n]_{\bullet}, B_{\bullet}) \rightarrow \widehat{\mathrm{Hom}}_{\mathrm{Ch}(\mathcal{C})}(A[n]_{\bullet}, B_{\bullet})$. By definition, this further descends to the homomorphism $\vartheta_n^0: \mathcal{E}xt_{\mathcal{C}}^n(A, B) \rightarrow \widehat{\mathcal{E}xt}_{\mathcal{C}}^n(A, B)$ as desired.
\end{proof}

From this proposition we deduce

\begin{lem}\label{lem:lesandcanmorphism}
The short exact sequence of chain complexes
\[  0 \rightarrow \mathrm{Bdd}_{\mathcal{C}}(A_{\bullet}, B_{\bullet})_{\bullet} \rightarrow \mathrm{Hyp}_{\mathcal{C}}(A_{\bullet}, B_{\bullet})_{\bullet} \rightarrow \mathrm{Vog}_{\mathcal{C}}(A_{\bullet}, B_{\bullet})_{\bullet} \rightarrow 0 \]
from the hypercohomology construction induces the long exact sequence
\[ \dots {} \rightarrow H^n(\mathrm{Bdd}_{\mathcal{C}}(A_{\bullet}, B_{\bullet})_{\bullet}) \rightarrow \mathrm{Ext}_{\mathcal{C}}^n(A, B) \xrightarrow{\Phi^n} \widehat{\mathrm{Ext}}_{\mathcal{C}}^n(A, B) 
\rightarrow H^{n+1}(\mathrm{Bdd}_{\mathcal{C}}(A_{\bullet}, B_{\bullet})_{\bullet}) \rightarrow {} \dots \]
where $\Phi^{\bullet}: \mathrm{Ext}_{\mathcal{C}}^{\bullet}(A, B) \rightarrow \widehat{\mathrm{Ext}}_{\mathcal{C}}^{\bullet}(A, B)$ is the canonical morphism from the definition of a Mislin completion. In particular, the terms $\Phi^n$ fit into a long exact sequence relating three distinct cohomological functors.
\end{lem}

\begin{rem}
This lemma is similar to Proposition~4.6 in S.\! Guo and L.\! Liang's paper~\cite{guo23}. Both results contain the same long exact sequence where our contribution lies in determining that it involves the terms of the canonical morphism to the Mislin completion.
\end{rem}

\section{An Eckmann--Shapiro Lemma and Dimension shifting}\label{sec:dimshifteckmannshapiro}

Complete cohomology of certain subgroups (objects) relates to the entire group (object) via an Eckmann--Shapiro Lemma, which can be used to establish a partial version of dimension shifting. Thus, let us establish the former, before establishing the latter.

\begin{lem}[Eckmann--Shapiro]\label{lem:eckmannshapiro}
Let $G$ be a group object, $H$ a subgroup object and $R$ a ring object in a category. Denote the abelian category of $R$-module objects with a compatible $G$-action by $Mod_R(G)$. Assume that any $M \in \mathrm{obj}(Mod_R(H))$ can be turned into an object in $Mod_R(G)$ by induction $\mathrm{Ind}_H^G(M)$ and coinduction $\mathrm{Coind}_H^G(M)$ while any $M \in \mathrm{obj}(Mod_R(G))$ can be turned into an object in $Mod_R(H)$ by restriction $\mathrm{Res}_H^G(M)$.
\begin{enumerate}
\item If the adjoint functors $\mathrm{Res}_H^G(-)$ and $\mathrm{Coind}_H^G(-)$ are exact and preserve projective objects, then
\[ \widehat{\mathrm{Ext}}_{R, H}^n(\mathrm{Res}_H^G(A), B) \cong \widehat{\mathrm{Ext}}_{R, G}^n(A, \mathrm{Coind}_H^G(B)) \]
as (unenriched) completed Ext-functors for every $n \in \mathbb{Z}$, $A \in \mathrm{obj}(Mod_R(G))$ and $B \in \mathrm{obj}(Mod_R(H))$. If $A = R$, one has
\[ \widehat{H}_R^n(H, B) \cong \widehat{H}_R^n(G, \mathrm{Coind}_H^G(B)) \, . \]
\item If the adjoint functors $\mathrm{Ind}_H^G(-)$ and $\mathrm{Res}_H^G(-)$ are exact and preserve projectives, then
\[ \widehat{\mathrm{Ext}}_{R, G}^n(\mathrm{Ind}_H^G(A), B) \cong \widehat{\mathrm{Ext}}_{R, H}^n(A, \mathrm{Res}_H^G(B)) \]
as (unenriched) completed Ext-functors for every $n \in \mathbb{Z}$, $A \in \mathrm{obj}(Mod_R(H))$ and $B \in \mathrm{obj}(Mod_R(G))$.
\end{enumerate}
\end{lem}

\begin{proof}
We only prove the first assertion as the second one is analogous. For this we use the resolution construction. If $f: A \rightarrow C$ is a morphism in $Mod_R(H)$, then it follows from the naturality of the adjunction and the Five Lemma~\cite[Proposition~I.1.1]{car56} that
\begin{align*}
\mathrm{Ker}\big(\mathrm{Hom}_{R, H}(\mathrm{Res}_H^G(f), B)\big) &\cong \mathrm{Ker}\big(\mathrm{Hom}_{R, G}(f, \mathrm{Coind}_H^G(B)\big) \text{ and} \\
\mathrm{Coker}\big(\mathrm{Hom}_{R, H}(\mathrm{Res}_H^G(f), B)\big) &\cong \mathrm{Coker}\big(\mathrm{Hom}_{R, G}(f, \mathrm{Coind}_H^G(B)\big) \, .
\end{align*}
To ease notation, we write $\mathrm{Ker}(\mathrm{Res}(f), B)$ for $\mathrm{Ker}\big(\mathrm{Hom}_{R, H}(\mathrm{Res}_H^G(f), B)\big)$, $\mathrm{Ker}(f, \mathrm{Cnd}(B))$ for $\mathrm{Ker}\big(\mathrm{Hom}_{R, H}(f, \mathrm{Coind}_H^G(B)\big)$ and analogously $\mathrm{Coker}(\mathrm{Res}(f), B)$, $\mathrm{Coker}(f, \mathrm{Cnd}(B))$. If $(A_l, a_l)_{l \in \mathbb{N}_0}$ is a projective resolution of $A$, then let $m, k \in \mathbb{N}_0$ with $m = n+k$ and consider the short exact sequence $0 \rightarrow \widetilde{B}_{k+1} \rightarrow B_k \rightarrow \widetilde{B}_k \rightarrow 0$. This gives rise to the diagram on the next page. The homomorphisms from the front to the back are isomorphisms arising from the above adjunction. By naturality of this adjunction and the universal property of kernels and cokernels, all squares of the diagram commute. Since we assume that restriction and coinduction are exact and preserve projectives, all rows are exact. Thus, the front and back side give rise to the connecting homomorphisms of the respective Ext-functors.\pagebreak

\begin{center}
\rotatebox{90}{%
\begin{tikzcd}[ampersand replacement = \&, column sep = tiny]
    \& \& \& \& \& \& {\scriptstyle \mathrm{Ext}_{R, G}^m(A{,} \, \mathrm{Cnd}_H^G(\widetilde{B}_k))} \arrow[dd] \& \\    
    \& \& \& \& \& {\scriptstyle \mathrm{Ext}_{R, H}^m(\mathrm{Res}_H^G(A){,} \, \widetilde{B}_k)} \arrow[ur] \arrow[dd] \& \\
    \& \& {\scriptstyle \mathrm{Coker}(a_m{,} \, \mathrm{Cnd}(\widetilde{B}_{k+1}))} \arrow[rr] \arrow[dd] \& \& {\scriptstyle \mathrm{Coker}(a_m{,} \, \mathrm{Cnd}(B_k))} \arrow[rr] \arrow[dd] \& \& {\scriptstyle \mathrm{Coker}(a_m{,} \, \mathrm{Cnd}(\widetilde{B}_k))} \arrow[r] \arrow[dd] \& {\scriptstyle 0} \\    
    \& {\scriptstyle \mathrm{Coker}(\mathrm{Res}(a_m){,} \, \widetilde{B}_{k+1})} \arrow[rr, crossing over] \arrow[ur] \arrow[dd] \& \& {\scriptstyle \mathrm{Coker}(\mathrm{Res}(a_m){,} \, B_k)} \arrow[rr, crossing over] \arrow[ur] \& \& {\scriptstyle \mathrm{Coker}(\mathrm{Res}(a_m){,} \, \widetilde{B}_k)} \arrow[r] \arrow[ur] \arrow[from = uu, crossing over] \& {\scriptstyle 0 \qquad \quad {}} \& \\
    \& {\scriptstyle {} \qquad \quad 0} \arrow[r] \& {\scriptstyle \mathrm{Ker}(a_{m+1}{,} \, \mathrm{Cnd}(\widetilde{B}_{k+1}))} \arrow[rr] \arrow[dd] \& \& {\scriptstyle \mathrm{Ker}(a_{m+1}{,} \, \mathrm{Cnd}(B_k))} \arrow[rr] \& \& {\scriptstyle \mathrm{Ker}(a_{m+1}{,} \, \mathrm{Cnd}(\widetilde{B}_k))} \& \\    
    {\scriptstyle 0} \arrow[r] \& {\scriptstyle \mathrm{Ker}(\mathrm{Res}(a_{m+1}){,} \, \widetilde{B}_{k+1})} \arrow[rr, crossing over] \arrow[ur] \arrow[dd] \& \& {\scriptstyle \mathrm{Ker}(\mathrm{Res}(a_{m+1}){,} \, B_k)} \arrow[rr] \arrow[ur] \arrow[from = uu, crossing over] \& \& {\scriptstyle \mathrm{Ker}(\mathrm{Res}(a_{m+1}){,} \, \widetilde{B}_k)} \arrow[ur] \arrow[from = uu, crossing over] \&\& \\
    \& \& {\scriptstyle \mathrm{Ext}_{R, G}^{m+1}(A{,} \,  \mathrm{Cnd}_H^G(\widetilde{B}_{k+1}))} \& \& \& \& \& \\
    \& {\scriptstyle \mathrm{Ext}_{R, H}^{m+1}(\mathrm{Res}_H^G(A){,} \, \widetilde{B}_{k+1})} \arrow[ur] \& \& \& \& \& \&
\end{tikzcd}
}
\end{center}\pagebreak

In particular, this diagram gives rise to the commuting square
\begin{center}
\begin{tikzcd}
    \mathrm{Ext}_{Mod_R(H)}^m(\mathrm{Res}_H^G(A){,} \, \widetilde{B}_k) \arrow[r, "\cong"] \arrow[d, "\delta^m"] & \mathrm{Ext}_{Mod_R(G)}^m(A{,} \, \mathrm{Coind}_H^G(\widetilde{B}_k)) \arrow[d, "\delta^m"] \\
    \mathrm{Ext}_{Mod_R(H)}^{m+1}(\mathrm{Res}_H^G(A){,} \, \widetilde{B}_{k+1}) \arrow[r, "\cong"] & \mathrm{Ext}_{Mod_R(H)}^{m+1}(A{,} \,  \mathrm{Coind}_H^G(\widetilde{B}_{k+1}))
\end{tikzcd}
\end{center}
These squares form a direct system in whose direct limit we obtain the desired isomorphism.
\end{proof}

To present examples where the above Eckmann--Shapiro Lemma applies, we explain based on~\cite[p.~62--63 and p.~67]{bro82} how induction, coinduction and restriction are usually defined for modules. Assume that there is a group ring (object) $R[G]$ such that the category of $R[G]$-module objects is equivalent to the category $Mod_R(G)$. Let
\[ - \otimes_{R[G]} -: Mod_R(G) \times Mod_R(G) \rightarrow Mod_R(G) \]
be a tensor product that is a right-adjoint to an internal Hom-functor
\[ \underline{\mathrm{Hom}}_{R, G}(-, -): Mod_R(G)^{\mathrm{op}} \times Mod_R(G) \rightarrow Mod_R(G) \, . \]
If $R[G]$ is an $(R[G], R[H])$-bimodule object, then we define restriction and coinduction as
\begin{align*}
\mathrm{Res}_H^G(-) := - \otimes_{R[G]} R[G]: \; \; &Mod_R(G) \rightarrow Mod_R(H) \text{ and} \\
\mathrm{Coind}_H^G(-) := \underline{\mathrm{Hom}}_{R, H}(R[G], -): \; \; &Mod_R(H) \rightarrow Mod_R(G) \, .
\end{align*}
If $R[G]$ is additionally an $(R[H], R[G])$-bimodule, then we define induction and restriction as
\begin{align*}
\mathrm{Ind}_H^G(-) := - \otimes_{R[H]} R[G]: \; \; &Mod_R(H) \rightarrow Mod_R(G) \text{ and} \\
\mathrm{Res}_H^G(-) := \underline{\mathrm{Hom}}_{R, G}(R[G], -): \; \; &Mod_R(G) \rightarrow Mod_R(H)
\end{align*}

\begin{exl}\label{exl:eckmannshapiro}
The conditions of Lemma~\ref{lem:eckmannshapiro} are satisfied in the following two instances.
\begin{enumerate}
\item $G$ is a discrete group, $H$ a finite index subgroup and $R$ a discrete commutative ring. Modules are taken over the respective group rings.
\item $G$ is a profinite group, $H$ an open subgroup and $R$ a profinite commutative ring. Profinite modules are taken over the respective completed group rings.
\end{enumerate}
\end{exl}

\begin{proof}
Let us first clarify some points regarding the profinite setting. Since the subgroup $H$ is open in $G$, it is of finite index~\cite[Lemma~0.3.1]{wil98}. Although there is no internal Hom-functor for profinite modules in general, coinduction and restriction can be nevertheless defined because $H$ is open subgroup of the profinite group $G$. Namely, for any $R{\llbracket}H{\rrbracket}$-module $M$ and any $R{\llbracket}G{\rrbracket}$-module $N$, endowing $\mathrm{Coind}_H^G(M) = \mathrm{Hom}_{R{\llbracket}H{\rrbracket}}(R{\llbracket}G{\rrbracket}, M)$ and $\mathrm{Res}_H^G(N) = \mathrm{Hom}_{R{\llbracket}G{\rrbracket}}(R{\llbracket}G{\rrbracket}, N)$ with the compact-open topology turns them into profinite modules~\cite[p.~369--371]{sym00}. As there is a tensor product for profinite modules~\cite[p.~177/191]{rib10}, induction and restriction can be defined in this case. By~\cite[Lemma~7.8]{bog16} restriction is left adjoint to coinduction and induction left adjoint to restriction in the profinite case. \\

As both the discrete and profinite case can be treated analogously from this point, we write $R[G]$ for either the discrete or competed group ring of $G$ over $R$ and denote by $H$ the finite index (open) subgroup. It follows from the description $\mathrm{Res}_H^G(-) = \mathrm{Hom}_{R[G]}(R[G], -)$ that restriction is exact. By~\cite[Proposition~I.3.1]{bro82}, \cite[Proposition~5.4.2]{rib10} and~\cite[Proposition~5.7.1]{rib10} the $R[H]$-module $R[G]$ is projective. We thus infer by~\cite[Lemma~2.2.3]{wei94} that coinduction is exact and by~\cite[Proposition~5.5.3]{rib10} and~\cite[p.~68]{wei94} that induction is exact. According to the proof of~\cite[Corollary~7.9]{bog16}, induction preserves projectives. Because the (open) subgroup $H$ is of finite index in $G$, $\mathrm{Ind}_H^G(M) \cong \mathrm{Coind}_H^G(M)$ for every $R[H]$-module $M$ by~\cite[Proposition~III.5.9]{bro82} and~\cite[p.~371]{sym00}. Due to this isomorphism and~\cite[Lemma~2.2.3]{wei94}, coinduction preserves projectives.
\end{proof}

\begin{thm}[Dimension shifting]\label{thm:dimshifting}
Let $T^{\bullet}: \mathcal{C} \rightarrow \mathcal{D}$ be a cohomological functor where $\mathcal{C}$ has enough projectives and in $\mathcal{D}$ all countable direct limits exist and are exact.
\begin{itemize}
\item For every $M \in \mathrm{obj}(\mathcal{C})$ there is $M^{\ast} \in \mathrm{obj}(\mathcal{C})$ such that $\widehat{T}^{n+1}(M^{\ast}) \cong \widehat{T}^n(M)$ for every $n \in \mathbb{Z}$.
\item If there is a monomorphism $f: M \rightarrow N$ in $\mathcal{C}$ with $\widehat{T}^k(N) = 0$ for every $k \in \mathbb{Z}$, then $\widehat{T}^{n-1}(\mathrm{Coker}(f)) \cong \widehat{T}^n(M)$.
\item Assume that $G$ is a group object, that $R$ is a ring object and that an Eckmann--Shapiro Lemma such as Lemma~\ref{lem:shapiro} holds. Then there is a monomorphism as in the previous assertion for $\widehat{T}^{\bullet} = \widehat{H}_R^{\bullet}(G, -)$ if there exists a subgroup object $H$ of finite cohomological dimension over $R$ such that $\mathrm{Res}_H^G$ is a faithful functor.
\end{itemize}
\end{thm}

\begin{proof}
For the first assertion consider the short exact sequence $0 \rightarrow \widetilde{M}_1 \rightarrow M_0 \rightarrow M \rightarrow 0$ and set $M^{\ast} := \widetilde{M}_1$. Then this assertion follows from Axiom~\ref{axm:les} and Definition~\ref{defn:mislincompletion}. The second assertion is deduced analogously. Regarding the third assertion, if we consider adjoint functors $L: \mathcal{C} \rightarrow \mathcal{D}$, $R: \mathcal{D} \rightarrow \mathcal{C}$, then one can deduce from~\cite[Theorem~IV.3.1]{mac78} that the left adjoint $L$ is faithful if and only if the unit over every $D \in \mathrm{obj}(\mathcal{D})$ is a monomorphism $D \rightarrow RL(D)$. Thus, the third assertion follows from the second one and Lemma~\ref{lem:prevanishing}.
\end{proof}

We use Example~\ref{exl:eckmannshapiro} to deduce the following result.

\begin{exl}\label{exl:dimshifting}
The conditions of the third assertion in Theorem~\ref{thm:dimshifting} are satisfied in the following two instances.
\begin{enumerate}
\item $G$ is a discrete group, $H$ a finite-index subgroup of finite cohomological dimension and $R$ a discrete commutative ring. Modules are taken to be discrete over the group ring $R[G]$.
\item $G$ is a profinite group, $H$ an open subgroup of finite cohomological dimension and $R$ a profinite commutative ring. Modules are taken to be profinite over the completed group ring $R{\llbracket}G{\rrbracket}$.
\end{enumerate}
\end{exl}

\section{External products of unenriched Ext-functors}\label{sec:ordinarycohomprods}

To construct external products and thus cup products of completed Ext-functors, we provide a category-theory-flavoured outline of these products for Ext-functors in this section. As we have not found a treatment of external products in categories of module objects in full generality, we generalise hereby K.\! S.\! Brown's account of cup products in group cohomology found in~\cite[Chapter~V]{bro82}. \\

The starting point of external products and thus of cup product are tensor products. Tensor products are taken to be bi-additive associate functors. Given our focus on group cohomology, we only consider tensor products $\otimes_R$ in categories of module objects $Mod_R$ over a ring object $R$ or in categories $Mod_R(G)$ of $R$-module objects with a compatible action of a group object $G$. In either case, it is assumed that this category has enough projectives and that the tensor product $P \otimes_R Q$ is projective whenever $P$, $Q$ are projective. Further, it is assumed that there are natural isomorphisms $M \otimes_R R \cong M \cong R \otimes_R M$.

\begin{exl}[Tensor products]\label{exl:tensorprods}
A tensor product $\otimes_R$ satisfies the above properties in the following instances
\begin{itemize}
\item $R$ is a commutative ring, $G$ is a group and modules are taken to be discrete.
\item $R$ is a commutative profinite ring, $G$ is a profinite group and modules are taken to be profinite.
\end{itemize}
\end{exl}

\begin{proof}
The case of rings and modules is classical. In the profinite case, the completed tensor product $\widehat{\otimes}_R$ is constructed in $Mod_R$ in~\cite[Section~5.5]{rib10} and in $Mod_R(G)$ in~\cite[Section~5.8]{rib10}. Bi-additivity and the isomorphisms $M \widehat{\otimes}_R R \cong M \cong R \widehat{\otimes}_R M$ are subject to~\cite[Proposition~5.5.3]{rib10}. Concerning associativity, $\widehat{\otimes}_R$ satisfies according to~\cite[p.~177]{rib10} the universal property that every continuous bilinear map $A \times B \rightarrow T$ factors uniquely through a continuous homomorphism $A \widehat{\otimes}_R B \rightarrow T$. Because $(A \widehat{\otimes}_R B) \widehat{\otimes}_R C$ and $A \widehat{\otimes}_R (B \widehat{\otimes}_R C)$ satisfy an analogous universal property for triples and continuous multilinear maps, they are isomorphic and $\widehat{\otimes}_R$ is thus associative. Lastly, any projective profinite module is a retract of a free profinite module. It follows from Exercise~5.5.5 and the proof of Proposition~5.8.3 in~\cite{rib10} that the tensor product of two projective profinite modules is again projective.
\end{proof}

\begin{notn}
For the reminder of the paper we write
\begin{itemize}
\item $Mod$ for $Mod_R$ or $Mod_{R}(G)$,
\item $Hom$ for $Hom_R$ or $Hom_{R, G}$ and
\item $\mathrm{Ext}^n(-, -)$ for $\mathrm{Ext}_R^n(-, -)$ or $\mathrm{Ext}_{R, G}^n(-, -)$.
\end{itemize}
\end{notn}

We extend $\otimes_R$ to a tensor product of projective resolutions. If $A$, $C$ are objects in $Mod$ and $(A_n, a_n)_{n \in \mathbb{N}_0}$, $(C_n,c _n)_{n \in \mathbb{N}_0}$ are corresponding projective resolutions, then we define the tensor product $A_{\bullet} \otimes_R C_{\bullet}$ of $A_{\bullet}$ and $C_{\bullet}$ as the following chain complex. For $B_1, \dots, B_n$ objects in $Mod$ there is a coproduct $\bigoplus_{i = 1}^n B_i$ whose canonical monomorphisms we denote by $i_{n, B_k}: B_k \rightarrow \bigoplus_{i = 1}^n B_i$ and canonical epimorphisms by $p_{n, B_k}: \bigoplus_{i = 1}^n B_i \rightarrow B_k$; see~\cite[p.~250--251]{mac95} for more detail. According to~\cite[p.~163]{mac95} we form for $n \in \mathbb{N}_0$ the $n$-chains as
\[ (A_{\bullet} \otimes_R C_{\bullet})_n := \bigoplus_{i = 0}^n A_i \otimes_R C_{n-i} \]
and the corresponding boundary map as
\begin{align}
D_{n+1} := &\sum_{k = 0}^{n+1}\big(i_{n+1, A_{k-1} \otimes_R C_{n+1-k}} \circ (a_k \otimes_R \mathrm{id}_{C_{n+1-k}}) \circ p_{n+2, A_k \otimes_R C_{n+1-k}} \nonumber  \\
&+ (-1)^k i_{n+1, A_k \otimes_R C_{n-k}} \circ (\mathrm{id}_{A_k} \otimes_R c_{n+1-k}) \circ p_{n+2, A_k \otimes_R C_{n+1-k}}\big): \label{eq:tensordifferential} \\
&(A_{\bullet} \otimes_R C_{\bullet})_{n+1} \rightarrow (A_{\bullet} \otimes_R C_{\bullet})_n \nonumber
\end{align}
where it is understood that $a_0 = c_0 = 0$. All terms of $A_{\bullet} \otimes_R C_{\bullet}$ are projective because we have assume that the tensor product of two projective modules is again projective. If $A_{\bullet}'$ and $C_{\bullet}'$ are different projective resolutions, then $A_{\bullet} \otimes_R C_{\bullet}$ is chain homotopic to $A_{\bullet}' \otimes_R C_{\bullet}'$~\cite[p.~164]{mac95}. Write $a: A_0 \rightarrow A$, $c: C_0 \rightarrow C$ for the augmentation maps. As can been inferred from~\cite[p.~164/229]{mac95}, one needs to impose conditions on $A_{\bullet} \otimes_R C_{\bullet}$ to ensure that it is an acyclic complex and thus a projective resolution of $A \otimes_R C$ with augmentation map $a \otimes_R c: A_0 \otimes_R C_0 \rightarrow A \otimes_R C$. \\

External products arise from defining tensor products of chain maps. For this let $(B_n, b_n)_{\mathbb{N}_0}$, $(E_n, e_n)_{n \in \mathbb{N}_0}$ be chain complexes that vanish in negative degrees. As at the end of Section~\ref{sec:constructions} when discussing the hypercohomology construction, we extend the projective resolutions $A_{\bullet}$, $C_{\bullet}$ to chain complexes indexed over $\mathbb{Z}$ by setting them to zero in negative degrees. For $m, n \in \mathbb{N}_0$ let $u_{\bullet} \in \mathrm{Hyp}_R(A_{\bullet}, B_{\bullet})_m$ and $v_{\bullet} \in \mathrm{Hyp}_R(C_{\bullet}, E_{\bullet})_n$ be cochains in the respective hypercohomology complex. Similar to the start of Subsection~6.1 in~\cite{ghe24}, we consider them as componentwise morphisms of the form $\lbrace u_{m+k}: A_{m+k} \rightarrow B_k \rbrace_{k \in \mathbb{Z}}$ and $\lbrace v_{n+l}: C_{n+l} \rightarrow E_l \rbrace_{l \in \mathbb{Z}}$ which do not need to be chain maps. According to~\cite[p.~10]{bro82}, defining
\begin{align}
(u_{\bullet} \otimes_R v_{\bullet})_k := \sum_{l = 0}^k (-1)^{(m+l)n} &i_{k, B_l \otimes_R E_{k-l}} \circ (u_{m+l} \otimes_R v_{n+k-l}) \circ p_{m+n+k, A_{m+l} \otimes_R C_{n+k-l}}: \nonumber \\
&(A_{\bullet} \otimes_R C_{\bullet})_{m+n+k} \rightarrow (B_{\bullet} \otimes_R E_{\bullet})_k \, . \label{eq:chaintensorprod}
\end{align}
for any $k \in \mathbb{N}_0$ yields a cochain $u_{\bullet} \otimes_R v_{\bullet}$ in $\mathrm{Hyp}_R(A_{\bullet} \otimes_R C_{\bullet}, B_{\bullet} \otimes_R E_{\bullet})_{m+n}$. Since $\otimes_R$ is bi-additive, this yields a homomorphism
\[ \mathrm{Hyp}_R(A_{\bullet}, B_{\bullet})_m \otimes \mathrm{Hyp}_R(C_{\bullet}, E_{\bullet})_n  \rightarrow \mathrm{Hyp}_R(A_{\bullet} \otimes_R C_{\bullet}, B_{\bullet} \otimes_R E_{\bullet})_{m+n}, \, u_{\bullet} \otimes v_{\bullet} \mapsto u_{\bullet} \otimes_R v_{\bullet} \]
where $\otimes$ denotes the tensor product of abelian groups. If we take the differential
\[ d^{m+n}: \mathrm{Hyp}_R(A_{\bullet} \otimes_R C_{\bullet}, B_{\bullet} \otimes_R E_{\bullet})_{m+n} \rightarrow \mathrm{Hyp}_R(A_{\bullet} \otimes_R C_{\bullet}, B_{\bullet} \otimes_R E_{\bullet})_{m+n+1} \]
as in Equation~\ref{eq:hypercohombd}, then it is also noted in~\cite[p.~10]{bro82} that
\begin{equation}\label{eq:wellbehaveddiffs}
d^{m+n}(u_{\bullet} \otimes_R v_{\bullet}) = (d^m u_{\bullet}) \otimes_R v_{\bullet} + (-1)^m u_{\bullet} \otimes_R (d^n v_{\bullet}) \, .
\end{equation}
Let $B, E \in \mathrm{obj}(\mathcal{C})$. Similar to the start of Subsection~6.1 in~\cite{ghe24}, we may consider the chain complex $i(B)_{\bullet}$ whose only nonzero terms is $B$ in degree $0$ and similarly $i(E)_{\bullet}$. If we define unenriched Ext-functors as derived functors of unenriched Hom-functors, then $\mathrm{Ext}^n(A, B) = H^n(\mathrm{Hyp}_R(A_{\bullet}, i(B)_{\bullet}))$ by definition of the corresponding differentials. We assume now explicitly that $A_{\bullet} \otimes_R C_{\bullet}$ is a projective resolution of $A \otimes_R C$. By this and Equation~\ref{eq:wellbehaveddiffs}, the tensor product of cochains as in Equation~\ref{eq:chaintensorprod} descends to a well defined homomorphism
\begin{align}
\vee: &\mathrm{Ext}^m(A, B) \otimes \mathrm{Ext}^n(C, E) \rightarrow \mathrm{Ext}^{m+n}(A \otimes_R C, B \otimes_R E) \nonumber \\
&\big(u + \mathrm{Im}(\mathrm{Hom}(a_m, B))\big) \otimes \big(v + \mathrm{Im}(\mathrm{Hom}(c_n, E))\big) \label{eq:ordinaryextprod} \\
&{} \qquad \mapsto (-1)^{mn} (u \otimes_R v) \circ p_{m+n, A_m \otimes_R C_n} + \mathrm{Im}(\mathrm{Hom}(D_{m+n}, B \otimes_R E)) \nonumber
\end{align}
that is termed an external product~\cite[p.~109--110]{bro82}. \\

We present some properties of external products of unenriched Ext-functors that we shall generalise to completed unenriched Ext-functors in Section~\ref{sec:cohomprodpropties}. To demonstrate that external products are natural as is mentioned in~\cite[p.~110]{bro82}, let $r: X \rightarrow A$, $s: Y \rightarrow C$, $f: B \rightarrow M$, $g: E \rightarrow N$ be morphisms in $Mod$. If we take corresponding projective resolutions, assume that $X_{\bullet} \otimes_R Y_{\bullet}$ is a projective resolution of $X \otimes_R Y$ and consider lifts to chain maps $r_{\bullet}: X_{\bullet} \rightarrow A_{\bullet}$, $s_{\bullet}: Y_{\bullet} \rightarrow C_{\bullet}$ as in the Comparison Theorem~\cite[Theorem~2.2.6]{wei94}. Due to Equation~\ref{eq:wellbehaveddiffs}, the chain map $r_{\bullet} \otimes_R s_{\bullet}: X_{\bullet} \otimes_R Y_{\bullet} \rightarrow A_{\bullet} \otimes_R C_{\bullet}$ only depends on $r$ and $s$ up to chain homotopy. Then the square
\begin{equation}\label{diag:extprodnat}
\begin{tikzcd}
    \mathrm{Ext}^m(A{,} \, B) \otimes \mathrm{Ext}^n(C{,} \, E) \arrow[r, "\vee"] \arrow[d, "\mathrm{Ext}^m(r{,} \, f) \otimes \mathrm{Ext}^n(s{,} \, g)"] & \mathrm{Ext}^{m+n}(A \otimes_R C{,} \, B \otimes_R E) \arrow[d, "\mathrm{Ext}^{m+n}(r \otimes_R s{,} \, f \otimes_R g)"] \\
    \mathrm{Ext}^m(X{,} \, M) \otimes \mathrm{Ext}^n(Y{,} \, N) \arrow[r, "\vee"] & \mathrm{Ext}^{m+n}(X \otimes_R Y{,} \, M \otimes_R N)
\end{tikzcd}
\end{equation}
commutes and external products are indeed natural. \\

As is proved in~\cite[p.~110--111]{bro82}, external products of Ext-functors respect connecting homomorphisms, which is important in the construction of external products of completed Ext-functors. More specifically, if $0 \rightarrow B'' \rightarrow B' \rightarrow B \rightarrow 0$ is a short exact sequence and $E$ is such that $0 \rightarrow B'' \otimes_R E \rightarrow B' \otimes_R E \rightarrow B \otimes_R E \rightarrow 0$ remains a short exact sequence, then the diagram
\begin{equation}\label{diag:extandconnfirst}
\begin{tikzcd}
    \mathrm{Ext}^m(A{,} \, B) \otimes \mathrm{Ext}^n(C{,} \, E) \arrow[r, "\vee"] \arrow[d, "\delta^m \otimes \mathrm{id}"] & \mathrm{Ext}^{m+n}(A \otimes_R C{,} \, B \otimes_R E) \arrow[d, "\delta^{m+n}"] \\
    \mathrm{Ext}^{m+1}(A{,} \, B'') \otimes \mathrm{Ext}^n(C{,} \, E) \arrow[r, "\vee"] & \mathrm{Ext}^{m+n+1}(A \otimes_R C{,} \, B'' \otimes_R E)
\end{tikzcd}
\end{equation}
commutes. On the other hand, if $0 \rightarrow E'' \rightarrow E' \rightarrow E \rightarrow 0$ is a short exact sequence and $B$ is such that $0 \rightarrow B \otimes_R E'' \rightarrow B \otimes_R E' \rightarrow B \otimes_R E \rightarrow 0$ remains a short exact sequence, then the diagram
\begin{equation}\label{diag:extandconnsecond}
\begin{tikzcd}
    \mathrm{Ext}^m(A{,} \, B) \otimes \mathrm{Ext}^n(C{,} \, E) \arrow[r, "\vee"] \arrow[d, "\mathrm{id} \otimes \delta^n"] & \mathrm{Ext}^{m+n}(A \otimes_R C{,} \, B \otimes_R E) \arrow[d, "(-1)^m \delta^{m+n}"] \\
    \mathrm{Ext}^m(A{,} \, B) \otimes \mathrm{Ext}^{n+1}(C{,} \, E'') \arrow[r, "\vee"] & \mathrm{Ext}^{m+n+1}(A \otimes_R C{,} \, B \otimes_R E'')
\end{tikzcd}
\end{equation}
commutes. \\

To conclude that external products are associative as mentioned in~\cite[p.~111]{bro82}, recall that the direct product $X \oplus Y$ is defined as a split extension $X \rightleftarrows X \oplus Y \rightleftarrows Y$. If we tensor with $Z$, we conclude that $\otimes_R$ distributes over finite products in the sense that there are isomorphisms
\begin{align*}
&(X \oplus Y) \otimes_R Z \rightarrow (X \otimes_R Z) \oplus (Y \otimes_R Z) \text{ and} \\
&X \otimes_R (Y \oplus Z) \rightarrow (X \otimes_R Y) \oplus (X \otimes_R Z) \, .
\end{align*}
Because $\otimes_R$ is associative by assumption, $A_{\bullet} \otimes_R C_{\bullet} \otimes_R F_{\bullet}$ is a well defined chain complex of projectives. If the latter is also a projective resolution of $A \otimes_R C \otimes_R F$, it follows from Equation~\ref{eq:ordinaryextprod} that external products are associative, meaning that
\begin{equation}\label{eq:extprodassoc}
\forall x \in \mathrm{Ext}_R^m(A{,} \, B), y \in \mathrm{Ext}^n(C{,} \, E), z \in \mathrm{Ext}^o(F{,} \, G): \: x \vee (y \vee z) = (x \vee y) \vee z \, .
\end{equation}
If the tensor product $\otimes_R$ is symmetric, then the external product $\vee$ is graded commutative which is demonstrated in~\cite[p.~111--112]{bro82}. More specifically, let $\mathrm{swap}: X \otimes_R Y \rightarrow Y \otimes_R X$ denote a bi-natural isomorphism. If $X_{\bullet}$, $Y_{\bullet}$ are projective resolutions of $X$ and $Y$ such that $X_{\bullet} \otimes_R Y_{\bullet}$ is a projective resolution of $X \otimes_R Y$, we define a chain isomorphism $\mathrm{swap}_{\bullet}: X_{\bullet} \otimes_R Y_{\bullet} \rightarrow Y_{\bullet} \otimes_R X_{\bullet}$ by
\begin{align}
\mathrm{swap}_k := \sum_{l = 0}^k (-1)^{l(k-l)} i_{k+1, Y_{k-l} \otimes_R X_l} \circ \mathrm{swap} &\circ p_{k+1, X_l \otimes_R Y_{k-l}}: \label{eq:swapchaintensorprod} \\
&(X_{\bullet} \otimes_R Y_{\bullet})_k \rightarrow (Y_{\bullet} \otimes_R X_{\bullet})_k \nonumber
\end{align}
for any $k \in \mathbb{Z}$. Then the diagram
\begin{equation}\label{diag:extprodcommity}
\begin{tikzcd}
    \mathrm{Hyp}_R(A_{\bullet}{,} \, B_{\bullet})_m \otimes \mathrm{Hyp}_R(C_{\bullet}{,} \, E_{\bullet})_n \arrow[r] \arrow[d, "(-1)^{mn}\mathrm{swap}"] & \mathrm{Hyp}_R(A_{\bullet} \otimes_R C_{\bullet}{,} \, B_{\bullet} \otimes_R E_{\bullet})_{m+n} \arrow[d, "\mathrm{Hyp}_R(\mathrm{swap}_{\bullet}{,} \, \mathrm{swap}_{\bullet})_{m+n}"] \\
    \mathrm{Hyp}_R(C_{\bullet}{,} \, E_{\bullet})_n \otimes \mathrm{Hyp}_R(A_{\bullet}{,} \, B_{\bullet})_m \arrow[r] & \mathrm{Hyp}_R(C_{\bullet} \otimes_R A_{\bullet}{,} \, E_{\bullet} \otimes_R B_{\bullet})_{m+n}
\end{tikzcd}
\end{equation}
commutes where the horizontal homomorphisms are taken as in Equation~\ref{eq:chaintensorprod}. Therefore, external products are graded commutative, meaning that
\begin{equation}\label{eq:extprodgradedcomm}
\forall x \in \mathrm{Ext}^m(A, B), y \in \mathrm{Ext}^n(C, E):
x \vee y = (-1)^{mn} y \vee x \, .
\end{equation}

In the case of group cohomology, external products give rise to cup products. For this we assume that the restriction functor $Mod_R(G) \rightarrow Mod_R$ forgetting the $G$-action on $R$-module objects preserves projectives and that $R$ as an object in $Mod_R$ is projective. It follows from Lemma~9.8.2 and Lemma~10.4.4 in~\cite{wil98} that for any two projective resolutions $R_{\bullet}$, $R_{\bullet}'$ of $R$ also $R_{\bullet} \otimes_R R_{\bullet}'$ is a projective resolution of $R$. Consequently, the external product from Equation~\ref{eq:ordinaryextprod} becomes the cup product
\begin{equation}\label{eq:ordinarycupprod}
\smile: H_R^m(G, B) \otimes H_R^n(G, E) \rightarrow H_R^{m+n}(G, B \otimes_R E) \, .
\end{equation}
All above mentioned properties of external products pass on to cup products, meaning that cup products are natural, associative, respect connecting homomorphisms and are graded commutative. In addition, there is a multiplicative unit for cup products which is shown in~\cite[p.~111]{bro82}. More specifically, if $1 \in H_R^0(G, R)$ denotes the element arising from the augmentation map $\varepsilon: R_0 \rightarrow R$, then
\[ \forall x \in H_R^n(G, M): x \smile 1 = x = 1 \smile x \, . \]
By construction, the cup product is bi-additive. Hence, if we define cup products with at least one element of negative degree to vanish, then the cup product turns $\bigoplus_{n \in \mathbb{Z}} H_R^n(G, R)$ into graded ring with identity $1$.

\section{Existence of cohomology products}\label{sec:cohomprodsexists}

In this section we construct external products and cup products of completed unenriched Ext-functors under specific conditions and provide examples. Moreover, we construct Yoneda products in full generality. \\

For clarity, we introduce the following notation. Reserving the letter `$D$' for boundary maps, we denote the variables of external products by $A$, $B$, $C$ and $E$. We write $G$ for a group object and $R$ for a ring object. As we have already used the letter `$I$' for index sets, we denote the variables of Yoneda products by $F$, $H$ and $J$ in order to distinguish them from the variables of external products. As most constructions of these cohomology products involve taking direct limits, we require the following result that we have not found in the literature in this manner.

\begin{prop}\label{prop:dirlimcommswithtensor}
Let $(M_i, m_i)_{i \in \mathbb{N}}$, $(N_i, n_i)_{i \in \mathbb{N}}$ be direct systems of abelian groups and denote by $\otimes$ the tensor product of abelian groups. Then
\[ \varinjlim_{i \in \mathbb{N}} (M_i \otimes N_i, m_i \otimes n_i) \cong (\varinjlim_{j \in \mathbb{N}} M_j) \otimes (\varinjlim_{k \in \mathbb{N}} N_k) \, . \]
\end{prop}

\begin{proof}
We demonstrate that $\varinjlim_{i \in \mathbb{N}} (M_i \otimes N_i, m_i \otimes n_i)$ satisfies the universal property of the tensor product $(\varinjlim_{j \in \mathbb{N}} M_j) \otimes (\varinjlim_{k \in \mathbb{N}} N_k)$. More specifically, the tensor product $A \otimes B$ has the universal property that every bilinear map $A \times B \rightarrow T$ factors uniquely through a homomorphism $A \otimes B \rightarrow T$. Denote by $M_i \times N_i$ the cartesian product and observe that the squares
\begin{equation}\label{diag:cartprodandtensor}
\begin{tikzcd}
    M_i \times N_i \arrow[r, "q_i"] \arrow[d, "m_i \times n_i"] & M_i \otimes N_i \arrow[d, "m_i \otimes n_i"] \\
    M_{i+1} \times N_{i+1} \arrow[r, "q_{i+1}"] & M_{i+1} \otimes N_{i+1}
\end{tikzcd}
\end{equation}
commute. Although $M_i \times N_i$ and $M_{i+1} \times N_{i+1}$ are abelian groups and $m_i \otimes n_i$ a homomorphism, we regard the former as sets and the latter as a function. In particular, Diagram~\ref{diag:cartprodandtensor} gives rise to a direct system of functions in whose direct limit we obtain
\begin{equation}\label{eq:cartprodandtensor}
q := \varinjlim_{i \in \mathbb{N}} q_i: \varinjlim_{\mathbf{Set}, i \in \mathbb{N}} (M_i \times N_i, m_i \times n_i) \rightarrow \varinjlim_{\mathbf{Set}, i \in \mathbb{N}} (M_i \otimes N_i, m_i \otimes n_i)
\end{equation}
where the latter is taken to be a direct limit of sets. By~\cite[\href{https://stacks.math.columbia.edu/tag/002W}{Tag 002W}]{stacks-project}, the former direct limit can be given by
\begin{equation}\label{eq:swapdirlimitwithprod}
\varinjlim_{\mathbf{Set}, i \in \mathbb{N}} (M_i \times N_i, m_i \times n_i) \cong (\varinjlim_{\mathbf{Set}, j \in \mathbb{N}} M_j) \times (\varinjlim_{\mathbf{Set}, k \in \mathbb{N}} N_k) \, .
\end{equation}
In particular, if $m_{i, \infty}: M_i \rightarrow (\varinjlim_{\mathbf{Set}, j \in \mathbb{N}} M_j)$ and $n_{i, \infty}: N_i \rightarrow (\varinjlim_{\mathbf{Set}, k \in \mathbb{N}} N_k)$ denote functions to the respective direct limit, then the function
\begin{equation}\label{eq:maptotensordirlimit}
m_{i, \infty} \times n_{i, \infty}: M_i \times N_i \rightarrow (\varinjlim_{\mathbf{Set}, j \in \mathbb{N}} M_j) \times (\varinjlim_{\mathbf{Set}, k \in \mathbb{N}} N_k)
\end{equation}
represents the function to the direct limit $\varinjlim_{\mathbf{Set}, i \in \mathbb{N}} (M_i \times N_i, m_i \times n_i)$. According to~\cite[\href{https://stacks.math.columbia.edu/tag/04AX}{Tag 04AX}]{stacks-project}, if $\sqcup$ denotes the disjoint union of sets, then the latter direct limit in Equation~\ref{eq:cartprodandtensor} can be constructed as
\begin{equation}\label{eq:dirlimitofsets}
\varinjlim_{\mathbf{Set}, i \in \mathbb{N}} (M_i \otimes N_i, m_i \otimes n_i) = \bigsqcup_{n \in \mathbb{N}} M_n \otimes N_n / \sim
\end{equation}
where $x_i \in M_i \otimes N_i \sim  y_j \in M_j \otimes N_j$ if there is $i, j \leq k$ such that $x_i$ and $y_j$ are mapped to the same element in $M_k \otimes N_k$. It follows from the proof of~\cite[Proposition~1.2.1]{rib10} and~\cite[\href{https://stacks.math.columbia.edu/tag/09WR}{Tag 09WR}]{stacks-project} that it can be endowed with the structure of an abelian group such that it forms the colimit $\varinjlim_{i \in \mathbb{N}} (M_i \otimes N_i, m_i \otimes n_i)$ of abelian groups and not just of sets. Same holds true for $\varinjlim_{i \in \mathbb{N}} M_i$ and $\varinjlim_{i \in \mathbb{N}} N_i$ from Equation~\ref{eq:swapdirlimitwithprod}. Hence, the function $q$ from Equation~\ref{eq:cartprodandtensor} can be written as
\[ q: (\varinjlim_{j \in \mathbb{N}} M_j) \times (\varinjlim_{k \in \mathbb{N}} N_k) \rightarrow \varinjlim_{i \in \mathbb{N}} (M_i \otimes N_i, m_i \otimes n_i) \, . \]
To prove that $q$ satisfies the universal property of the tensor product, consider a bilinear function $a: (\varinjlim_{j \in \mathbb{N}} M_j) \times (\varinjlim_{k \in \mathbb{N}} N_k) \rightarrow A$ for an abelian group $A$. Composing with the homomorphisms $m_{i, \infty}$ and $n_{i, \infty}$ occurring in Equation~\ref{eq:maptotensordirlimit} yields a bilinear function
\[ M_i \times N_i \xrightarrow{m_{i, \infty} \times n_{i, \infty}} (\varinjlim_{j \in \mathbb{N}} M_j) \times (\varinjlim_{k \in \mathbb{N}} N_k) \xrightarrow{a} A \, . \]
By the universal property of the tensor product, there exists a unique homomorphism $b_i: M_i \otimes N_i \rightarrow A$ such that the square
\begin{equation}\label{diag:factorbytensorindirsystem}
\begin{tikzcd}
    M_i \times N_i \arrow[r, "q_i"] \arrow[d, "m_{i, \infty} \times n_{i, \infty}"] & M_i \otimes N_i \arrow[d, "b_i"] \\
    (\varinjlim_{j \in \mathbb{N}} M_j) \times (\varinjlim_{k \in \mathbb{N}} N_k) \arrow[r, "a"] & A
\end{tikzcd}
\end{equation}
commutes. By this and Diagram~\ref{diag:cartprodandtensor} we infer that the triangle
\begin{equation}\label{diag:getuniquetensormap}
\begin{tikzcd}
    M_i \otimes N_i \arrow[dr, bend left = 10, "b_i"] \arrow[d, "m_i \otimes n_i"] & \\
    M_{i+1} \otimes N_{i+1} \arrow[r, "b_{i+1}" near start] & A
\end{tikzcd}
\end{equation}
is commutative whence there is a homomorphism
\[ b := \varinjlim_{i \in \mathbb{N}} b_i: \varinjlim_{i \in \mathbb{N}} (M_i \otimes N_i, m_i \otimes n_i) \rightarrow A \]
By Diagram~\ref{diag:cartprodandtensor}, \ref{diag:factorbytensorindirsystem} and~\ref{diag:getuniquetensormap} we obtain the factorisation
\[ a: \varinjlim_{\mathbf{Set}, i \in \mathbb{N}} (M_i \times N_i, m_i \times n_i) \xrightarrow{q} \varinjlim_{\mathbf{Set}, i \in \mathbb{N}} (M_i \otimes N_i, m_i \otimes n_i) \xrightarrow{b} A \, . \]
Because $b$ is unique in this factorisation due to Equation~\ref{eq:dirlimitofsets} and the uniqueness of the homomorphisms $b_i$ from Diagram~\ref{diag:factorbytensorindirsystem} and~\ref{diag:getuniquetensormap}, $\varinjlim_{i \in \mathbb{N}} (M_i \otimes N_i, m_i \otimes n_i)$ satisfies the universal property of the tensor product $(\varinjlim_{j \in \mathbb{N}} M_j) \otimes (\varinjlim_{k \in \mathbb{N}} N_k)$.
\end{proof}

\begin{thm}\label{thm:extandcupprod}
Let $A_{\bullet}$, $C_{\bullet}$ be projective resolutions of objects $A$, $C$ in $Mod$ such that  $A_{\bullet} \otimes_R C_{\bullet}$ is a projective resolution of $A \otimes_R C$. Let $B_{\bullet}$, $E_{\bullet}$ be projective resolutions of $B$, $E$ such that for their syzygies $\widetilde{F}_k \in \lbrace \widetilde{B}_k$, $\widetilde{E}_k \rbrace$ the functors
\[ - \otimes_R \widetilde{F}_k, \widetilde{F}_k \otimes_R -: Mod \rightarrow Mod \]
are exact and preserve projectives. Assume that the restriction functor $Mod_R(G) \rightarrow Mod_R$ forgetting the $G$-action on $R$-module objects preserves projectives and that $R$ as an object in $Mod_R$ is projective.
\begin{enumerate}
    \item Then for every $m, n \in \mathbb{Z}$ external products
    \[ \vee: \widehat{\mathrm{Ext}}^m(A, B) \otimes \widehat{\mathrm{Ext}}^n(C, E) \rightarrow \widehat{\mathrm{Ext}}^{m+n}(A  \otimes_R C, B \otimes_R E) \]
    can be defined for completed unenriched Ext-functors equivalently through the resolution construction and the hypercohomoology construction.
    \item These external products descend to cup products
    \[ \smile: \widehat{H}_R^m(G, M) \otimes \widehat{H}_R^n(G, N) \rightarrow \widehat{H}_R^{m+n}(G, M \otimes_R N) \]
    of completed unenriched group cohomology as the tensor product of any two projective resolutions of $R$ remains a projective resolution of $R$
\end{enumerate}
\end{thm}

\begin{proof}
Because the existence of cup products follows from the existence of external products as at the end of Section~\ref{sec:ordinarycohomprods}, we only prove the latter where we first use the resolution construction. Since we need to consider direct limit systems for this, let $k, l \in \mathbb{N}_0$ and write $K := m+k$ and $L := n+l$. Then by Diagram~\ref{diag:extandconnfirst}, Diagram~\ref{diag:extandconnsecond} and~\cite[Proposition~III.4.1]{car56} we see that the cube on the next page commutes. As in~\cite[Definition~4.7]{ghe24}, one can construct for any sequence $o \in \lbrace 0, 1 \rbrace^{\mathbb{N}}$ a direct system
\[ \mathrm{Ext}^{m+P(o)_k}(A{,} \, \widetilde{B}_{P(o)_k}) \otimes \mathrm{Ext}^{n+D(o)_k}(C{,} \, \widetilde{E}_{D(o)_k}))_{k \in \mathbb{N}_0} \]
whose homomorphisms from one term to the next are given by
\[ \begin{cases} \delta^{m+{P(o)_k}} \otimes \mathrm{id} &\text{if } P(o)_{k+1} = P(o)_k+1 \\ \mathrm{id} \otimes \delta^{n+D(o)_k} &\text{if } P(o)_{k+1} = P(o)_k \end{cases} \, . \]
According to the below cube, this gives rise to a homomorphism
\begin{align}
\varinjlim_{k \in \mathbb{N}_0} \vee: \varinjlim_{k \in \mathbb{N}_0} \big(\mathrm{Ext}^{m+P(o)_k}(A, \widetilde{B}_{P(o)_k}) \otimes &\mathrm{Ext}^{n+D(o)_k}(C, \widetilde{E}_{D(o)_k})\big) \label{eq:dirlimforext} \\
&\rightarrow \varinjlim_{k \in \mathbb{N}_0} \mathrm{Ext}^{m+n+k}(A \otimes_R C, \widetilde{B}_{P(o)_k} \otimes_R \widetilde{E}_{D(o)_k}) \nonumber
\end{align}
in the direct limit. If the sequence $o$ takes the value $0$ only finitely many times, then there would be $d \in \mathbb{N}_0$ such that the above homomorphism would be of the form
\[ \varinjlim_{k \in \mathbb{N}_0} \vee: \mathrm{Ext}^{m+d}(A, \widetilde{B}_d) \otimes \widehat{\mathrm{Ext}}^n(C, E) \rightarrow \varinjlim_{k \in \mathbb{N}_0} \mathrm{Ext}^{m+n+k}(A \otimes_R C, \widetilde{B}_{P(o)_k} \otimes_R \widetilde{E}_{D(o)_k}) \, . \]
Since we do not wish to consider this to be an external product, the values $0$ and $1$ occur infinitely often in $o$. In particular, the changing signs arising from the homomorphisms on the right hand of Diagram~\ref{diag:extprodcube} cancel each other in the direct limit. Because we assume that tensoring with the sygyzies of the projective resolutions $B_{\bullet}$, $E_{\bullet}$ is exact and preserves projectives, the right hand direct limit in Equation~\ref{eq:dirlimforext} equals $\widehat{\mathrm{Ext}}^{m+n}(A \otimes_R C, B \otimes_R E)$ by the resolution construction. The left hand direct limit in Equation~\ref{eq:dirlimforext} is independent of the choice of the sequence $o$ according to the proof of~\cite[Theorem~4.15]{ghe24} and Diagram~\ref{diag:extprodcube}. Thus, if we choose $o$ to alternate between $0$ and $1$, then Equation~\ref{eq:dirlimforext} becomes
\[ \varinjlim_{k \in \mathbb{N}_0} \vee: \varinjlim_{k \in \mathbb{N}_0} \big(\mathrm{Ext}^{m+k}(A, \widetilde{B}_k) \otimes \mathrm{Ext}^{n+k}(C, \widetilde{E}_k)\big) \rightarrow \widehat{\mathrm{Ext}}^{m+n}(A \otimes_R C, B \otimes_R E) \, . \]\pagebreak

\begin{samepage}
\begin{center}
\rotatebox{90}{%
\begin{tikzcd}[ampersand replacement = \&, column sep = tiny]
  \& {\scriptstyle \mathrm{Ext}^K(A{,} \, \widetilde{B}_k) \otimes \mathrm{Ext}^L(C{,} \, \widetilde{E}_l)} \arrow[dl, "{\scriptstyle \delta^K \otimes \mathrm{id}}"] \arrow[rr, "{\scriptstyle \vee}" near end] \arrow[dddd, "{\scriptstyle \mathrm{id} \otimes \delta^L}"]
    \& \& {\scriptstyle \mathrm{Ext}^{K+L}(A \otimes_R C{,} \, \widetilde{B}_k \otimes_R \widetilde{E}_l)} \arrow[dl, "{\scriptstyle \delta^{K+L}}"] \arrow[dddd, "{\scriptstyle (-1)^K \delta^{K+L}}"] \\
  {\scriptstyle \mathrm{Ext}^{K+1}(A{,} \, \widetilde{B}_{k+1}) \otimes \mathrm{Ext}^L(C{,} \, \widetilde{E}_l)} \arrow[rr, crossing over, "{\scriptstyle \vee}" near end] \arrow[dddd, "{\scriptstyle \mathrm{id} \otimes \delta^L}"]
    \& \& {\scriptstyle \mathrm{Ext}^{K+L+1}(A \otimes_R C{,} \, \widetilde{B}_{k+1} \otimes_R \widetilde{E}_l)} \\ \\ \\
  \& {\scriptstyle \mathrm{Ext}^K(A{,} \, \widetilde{B}_k) \otimes \mathrm{Ext}^{L+1}(C{,} \, \widetilde{E}_{l+1})} \arrow[dl, "{\scriptstyle \delta^K \otimes \mathrm{id}}" near start] \arrow[rr, "{\scriptstyle \vee}" near end]
    \& \& {\scriptstyle \mathrm{Ext}^{K+L+1}(A \otimes_R C{,} \, \widetilde{B}_k \otimes_R \widetilde{E}_{l+1})} \arrow[dl, "{\scriptstyle \delta^{K+L+1}}"] \\
  {\scriptstyle \mathrm{Ext}^{K+1}(A{,} \, \widetilde{B}_{k+1}) \otimes \mathrm{Ext}^{L+1}(C{,} \, \widetilde{E}_{l+1})} \arrow[rr, "{\scriptstyle \vee}" near end]
    \& \& {\scriptstyle \mathrm{Ext}^{K+L+2}(A \otimes_R C{,} \, \widetilde{B}_{k+1} \otimes_R \widetilde{E}_{l+1})} \arrow[from=uuuu, crossing over, "{\scriptstyle (-1)^{K+1} \delta^{K+L+1}}"] \\
\end{tikzcd}
}
\end{center}
{\vspace{-110mm}
\begin{equation}\label{diag:extprodcube}
{} \quad {}
\end{equation}
${} \quad {}$ \vspace{100mm}}
\end{samepage}

By Proposition~\ref{prop:dirlimcommswithtensor}, this results in the desirerd external product
\[ \vee := \varinjlim_{k \in \mathbb{N}_0} \vee: \widehat{\mathrm{Ext}}^m(A, B) \otimes \widehat{\mathrm{Ext}}^n(C, E) \rightarrow \widehat{\mathrm{Ext}}^{m+n}(A \otimes_R C, B \otimes_R E) \, . \]
In order to translate this to the hypercohomology construction, consider two almost chain maps $\varphi_{\bullet+m}: A[m]_{\bullet} \rightarrow B_{\bullet}$ and $\psi_{\bullet+n}: C[n]_{\bullet} \rightarrow E_{\bullet}$. For $k \in \mathbb{N}_0$ write $\pi_k: B_k \rightarrow \widetilde{B}_k$ for the morphism to the $k^{\text{th}}$ syzygy and define $\varphi_{m+k}' := \pi_k \circ \varphi_{m+k}: A_{m+k} \rightarrow \widetilde{B}_k$ where we define $\psi_{n+k}': C_{n+k} \rightarrow \widetilde{E}_k$ analogously. According to~\cite[Lemma~6.7]{ghe24} and the proof of~\cite[Lemma~6.12]{ghe24}, there is $\kappa \in \mathbb{N}_0$ such that $(\varphi_{m+k})_{k \geq 2{\kappa}}$ is a chain map that gives rise to the element $\varphi_{m+2{\kappa}}'+\mathrm{Im}(\mathrm{Hom}(a_{m+2{\kappa}}, \widetilde{B}_{2{\kappa}})$ in $\mathrm{Ext}^{m+2{\kappa}}(A, \widetilde{B}_{2{\kappa}})$. We see by the proofs of~\cite[Lemma~6.7]{ghe24} and of~\cite[Lemma~6.12]{ghe24} that
\[ \delta^{m+K}\big(\varphi_{m+K}'+\mathrm{Im}(\mathrm{Hom}_{\mathcal{C}}(a_{m+K}, \widetilde{B}_K) \big) = \varphi_{m+K+1}'+\mathrm{Im}(\mathrm{Hom}_{\mathcal{C}}(a_{m+K+1}, \widetilde{B}_{K+1}) \]
for any $K \geq 2{\kappa}$. Analogously, there is $\lambda \in \mathbb{N}_0$ such that $(\psi_{n+l})_{l \geq 2{\lambda}}$  is a chain map that gives rise to the element $\psi_{n+2{\lambda}}'+\mathrm{Im}(\mathrm{Hom}_{\mathcal{C}}(c_{n+2{\lambda}}, \widetilde{E}_{2{\lambda}})$ in $\mathrm{Ext}_R^{n+2{\lambda}}(C, \widetilde{E}_{2{\lambda}})$. Given $K \geq 2{\kappa}$, $L \geq 2{\lambda}$, the external product of $(\varphi_{m+k})_{k \geq K}$ and $(\psi_{n+l})_{l \geq L}$ arises from
\[ \varphi_{m+K}' \otimes_R \psi_{n+L}': A_{m+K} \otimes_R C_{n+L} \rightarrow \widetilde{B}_K \otimes_R \widetilde{E}_L \]
according to Equation~\ref{eq:ordinaryextprod}. Observe that the diagrams
\begin{equation}\label{diag:liftmorphismone}
\begin{tikzcd}
    A_{m+K} \otimes_R C_{n+L} \arrow[rr, "\varphi_{m+K} \otimes_R \psi_{n+L}'"] \arrow[ddrr, "\varphi_{m+K}' \otimes_R \psi_{n+L}'"] & & B_K \otimes_R \widetilde{E}_L \arrow[dd, "\pi_K \otimes_R \mathrm{id}"] \\ \\
    & & \widetilde{B}_K \otimes_R \widetilde{E}_L
\end{tikzcd}
\end{equation}
and
\begin{equation}\label{diag:liftmorphismtwo}
\begin{tikzcd}
    A_{m+K} \otimes_R C_{n+L} \arrow[rr, "\varphi_{m+K}' \otimes_R \psi_{n+L}"] \arrow[ddrr, "\varphi_{m+K}' \otimes_R \psi_{n+L}'"] & & \widetilde{B}_K \otimes_R E_L \arrow[dd, "\mathrm{id} \otimes_R \pi_L"] \\ \\
    & & \widetilde{B}_K \otimes_R \widetilde{E}_L
\end{tikzcd}
\end{equation}
commute. In order that the external product of $\varphi_{\bullet+m}$ and $\psi_{\bullet+n}$ is again an almost chain map modulo chain homotopy, a choice of a projective resolution of $B \otimes_R E$ is required. Define the projective resolution $B_{\bullet} \otimes_R^o E_{\bullet}$ by setting for any $k \in \mathbb{N}_0$ the $k^{\text{th}}$ term to be
\[ (B_{\bullet} \otimes_R^o E_{\bullet})_k := \begin{cases} B_{P(o)_k} \otimes \widetilde{E}_{D(o)_k} &\text{if } P(o)_{k+1} = P(o)_k +1 \\
\widetilde{B}_{P(o)_k} \otimes E_{D(o)_k} &\text{if } P(o)_{k+1} = P(o)_k \end{cases} \]
and extend it to be zero in negative degrees. We define the componentwise morphism
\[ \varphi_{\bullet+m} \otimes_R^o \psi_{\bullet+n}: (A_{\bullet} \otimes_R C_{\bullet})[m+n]_{\bullet} \rightarrow B_{\bullet} \otimes_R^o E_{\bullet} \]
by setting the $k^{\text{th}}$ degree correspondingly to
\begin{align*}
\varphi_{m+P(o)_k} \otimes_R \psi_{n+D(o)_k}' \circ &p_{k+1, A_{m+P(o)_k} \otimes_R C_{n+D(o)_k}}: \\
&(A_{\bullet} \otimes_R C_{\bullet})_{m+n+k} \rightarrow B_{P(o)_k} \otimes_R \widetilde{E}_{D(o)_k} \text{ or} \\
\varphi_{m+P(o)_k}' \otimes_R \psi_{n+D(o)_k} \circ &p_{k+1, A_{m+P(o)_k} \otimes_R C_{n+D(o)_k}}: \\
&(A_{\bullet} \otimes_R C_{\bullet})_{m+n+k} \rightarrow \widetilde{B}_{P(o)_k} \otimes_R E_{D(o)_k}
\end{align*}
whenever $m+P(o)_k, n+D(o)_k \geq 0$ and to be zero otherwise. Due to the construction of external products through the resolution construction, Diagram~\ref{diag:liftmorphismone} and Diagram~\ref{diag:liftmorphismtwo}, $\varphi_{\bullet+m} \otimes_R^o \psi_{\bullet+n}$ is a chain map in degrees $k$ for which $P(o)_k \geq 2{\kappa}$, $D(o)_k \geq 2{\lambda}$ and thus an almost chain map. We conclude that
\begin{align*}
\widehat{\mathrm{Ext}}^m &(A, B) \otimes \widehat{\mathrm{Ext}}^n(C, E) \rightarrow \widehat{\mathrm{Ext}}^{m+n}(A \otimes_R C, B \otimes_R E) \\
(\varphi_{\bullet+m}+\widehat{\mathrm{Null}}&(A[m]_{\bullet}, B_{\bullet})) \otimes (\psi_{\bullet+n}+\widehat{\mathrm{Null}}(C[n]_{\bullet}, E_{\bullet})) \\
&\mapsto (\varphi_{\bullet+m} \otimes_R^o \psi_{\bullet+n}+\widehat{\mathrm{Null}}((A_{\bullet} \otimes_R C_{\bullet})[m+n]_{\bullet}, B_{\bullet} \otimes_R^o E_{\bullet})
\end{align*}
is the external product given by the hypercohomology construction.
\end{proof}

\begin{rem}
External products for completed Ext-functors cannot be defined through tensor products of almost chain maps as in Equation~\ref{eq:chaintensorprod} because this does not yield an almost chain map in general~\cite[p.~110]{ben92}.
\end{rem}

\begin{exl}\label{exl:extandcupprods}
All conditions of Theorem~\ref{thm:extandcupprod} are satisfied in the following instances.
\begin{enumerate}
    \item $G$ is a discrete group, $R$ a principal ideal domain and the restriction of the $R[G]$-modules $A$, $B$, $C$, $E$ to $R$-modules is projective.
    \item $G$ is a profinite group, $R$ a profinite commutative ring with a unique maximal open ideal and the restriction of the profinite $R{\llbracket}G{\rrbracket}$-modules $A$, $B$, $C$, $E$ to $S$-modules is projective. The $p$-adic integers $\mathbb{Z}_p$ are an example of such a profinite ring where the restriction of any $p$-torsionfree profinite $\mathbb{Z}_p{\llbracket}G{\rrbracket}$-module to a $\mathbb{Z}_p$-module is projective.
\end{enumerate}
\end{exl}

\begin{proof}
As both assertions are proved analogously, we write $R[G]$ for either the discrete or completed group ring of $G$ over $R$. Namely, the category $Mod_R(G)$ is equivalent to the category $Mod_{R[G]}$ of discrete/profinite $R[G]$-module where one invokes \cite[Proposition~5.3.6]{rib10} in the profinite case. By Example~\ref{exl:tensorprods}, there are tensor products satisfying the preconditions to construct external products for completed Ext-functors and thus cup products for complete cohomology. Since the conditions of Theorem~\ref{thm:extandcupprod} involve projectives, we note that every projective module can be realised as a retract of a free module where the profinite case is due to~\cite[Proposition~5.4.2]{rib10}. Any free $R[G]$-module is free as an $R$-module where the profinite case is by~\cite[Corollary~5.7.2]{rib10}. Thus, the restriction functor $Mod_R(G) \rightarrow Mod_R$ preserves projectives where $R$ as an $R$-module is projective. \\

Denote by $F$ any of $A$, $B$, $C$ or $E$. Because the restriction functor preserves projectives and is exact for being a forgetful functor, any projective resolution $F_{\bullet}$ of $F$ as an $R[G]$-module  is also a projective resolution of $R$-modules. As $F$ is projective as an $R$-module, the short exact sequence $0 \rightarrow \widetilde{F}_1 \rightarrow F_0 \rightarrow F \rightarrow 0$ is split and hence $\widetilde{F}_1$ projective as an $R$-module. Inductively, we conclude for every $k \in \mathbb{N}_0$ that $\widetilde{F}_k$ is projective as an $R$-module. According to~\cite[p.~29]{bro82} and~\cite[Proposition~5.5.3]{rib10}, the functors
\[ - \otimes_R \widetilde{F}_k, \widetilde{F}_k \otimes_R -: Mod_R(G) \rightarrow Mod_R(G) \]
are exact. According to~\cite[Theorem~9.8]{rot02} and~\cite[Proposition~7.5.1]{wil98}, $\widetilde{F}_k$ is also free as an $R$-module. It then follows from the proof of~\cite[Proposition~3.3.2]{sym00} that the above functors preserve projectives. According to~\cite[Lemma~10.4.4]{wil98}, $A_{\bullet} \otimes_R C_{\bullet}$ is a projective resolution of $A \otimes_R C$.
\end{proof}

\begin{rem}
External products of completed Ext-functors and thus cup products of complete cohomology cannot be readily established in the context of condensed mathematics. Namely, the tensor product of two condensed projective modules need not be projective according to~\cite[Proposition~3.7]{sch22}. Even if one considers solid modules, which are condensed modules satisfying a form of completion, tensoring with the projective solid $\underline{\mathbb{Z}}$-module $\prod_{2^{2^{\aleph_0}}} \underline{\mathbb{Z}}$ is not exact according to A.\! I.\! Efimov (yet to appear).
\end{rem}

In contrast to external products and cup products, we construct Yoneda products in full generality. To this end, recall that for every $n \in \mathbb{Z}$ the completed unenriched Ext-functor $\widehat{\mathrm{Ext}}_{\mathcal{C}}^n(-, -): \mathcal{C}^{\mathrm{op}} \times \mathcal{C} \rightarrow \mathbf{Ab}$ forms a bifunctor that is additive in both variables by~\cite[Proposition~5.8]{ghe24} and~\cite[Proposition~6.5]{ghe24}. We generalise hereby the constructions of Yoneda products found in~\cite[p.~110]{ben92} and prove that they are equivalent.

\begin{thm}\label{thm:yonedaprods}
Let $F, H, J \in \mathrm{obj}(\mathcal{C})$. If $\otimes$ denotes the tensor product in $\mathbf{Ab}$, then for every $m, n \in \mathbb{Z}$ Yoneda products
    \[ \circ: \widehat{\mathrm{Ext}}_{\mathcal{C}}^n(H, J) \otimes \widehat{\mathrm{Ext}}_{\mathcal{C}}^m(F, H) \rightarrow \widehat{\mathrm{Ext}}_{\mathcal{C}}^{m+n}(F, J) \]
    can be defined for completed unenriched Ext-functors equivalently by the hypercohomology construction as composition of almost chain maps or by the naïve construction as a direct limit of the composition functors of the functors $[-, -]_{\mathcal{C}}$ from~\cite[Proposition~5.1]{ghe24}.
\end{thm}

\begin{proof}
The definition of Yoneda products through the hypercohomology construction is analogous to their definition for unenriched Ext-functors, which is covered in~\cite[p.~166]{gel03} for instance. Namely, if the almost chain map $f[n]_{\bullet+m}: F[m+n]_{\bullet} \rightarrow H[n]_{\bullet}$ is a representative an element in $\widehat{\mathrm{Ext}}_{\mathcal{C}}^m(F, H)$ and $g_{\bullet+n}: H[n]_{\bullet} \rightarrow J_{\bullet}$ a representative of an element in $\widehat{\mathrm{Ext}}_{\mathcal{C}}^n(H, J)$, then their composition $g_{\bullet+n} \circ f[n]_{\bullet+m}: F[m+n]_{\bullet} \rightarrow J_{\bullet}$ is again an almost chain map. If $f_{\bullet+m}'$ is chain homotopic to $f_{\bullet+m}$ and $g_{\bullet+n}'$ chain homotopic to $g_{\bullet+n}$, then $g_{\bullet+n} \circ f[n]_{\bullet+m}$ is chain homotopic to $g_{\bullet+n} \circ f'[n]_{\bullet+m}$ which is in turn chain homotopic to $g_{\bullet+n}' \circ f'[n]_{\bullet+m}$. Hence,
\[ g_{\bullet+n} \circ f[n]_{\bullet+m}+\widehat{\mathrm{Null}}(F[m+n]_{\bullet}, J_{\bullet}) \]
is a well defined element in $\widehat{\mathrm{Ext}}_{\mathcal{C}}^{m+n}(F, J)$. Since this operation is bi-additive, it constitutes a Yoneda product
\[ \widehat{\mathrm{Ext}}_{\mathcal{C}}^n(H, J) \otimes \widehat{\mathrm{Ext}}_{\mathcal{C}}^m(F, H) \rightarrow \widehat{\mathrm{Ext}}_{\mathcal{C}}^{m+n}(F, J) \, . \]
Moving over to the naïve construction, one can see in~\cite[Proposition~5.1]{ghe24} that for any $k \in \mathbb{N}_0$ with $n+m+k, n+k \geq 0$ the composition functor of morphisms descends to a bifunctor
\[ \circ: [\widetilde{H}_{n+k}, \widetilde{J}_k]_{\mathcal{C}} \times [\widetilde{F}_{n+m+k}, \widetilde{H}_{n+k}]_{\mathcal{C}} \rightarrow [\widetilde{F}_{n+m+k}, \widetilde{J}_k]_{\mathcal{C}} \, . \]
In particular, if we take an element $f_{n+m+k}' + \mathcal{P}_{\mathcal{C}}(\widetilde{F}_{n+m+k}, \widetilde{H}_{n+k})$ in $[\widetilde{F}_{n+m+k}, \widetilde{H}_{n+k}]_{\mathcal{C}}$ and an element $g_{n+k}'+\mathcal{P}_{\mathcal{C}}(\widetilde{H}_{n+k}, \widetilde{J}_k)$ in $[\widetilde{H}_{n+k}, \widetilde{J}_k]_{\mathcal{C}}$, then they give rise to a well defined element $g_{n+k}' \circ f_{n+m+k}'+ \mathcal{P}_{\mathcal{C}}(\widetilde{F}_{n+m+k}, \widetilde{J}_k)$ in $[\widetilde{F}_{n+m+k}, \widetilde{J}_k]_{\mathcal{C}}$. By~\cite[Proposition~5.2]{ghe24} and~\cite[Definition~5.3]{ghe24} we can pass to the direct limit to obtain
\[ \varinjlim_{k \in \mathbb{N}_0} \big([\widetilde{H}_{n+k}, \widetilde{J}_k]_{\mathcal{C}} \times [\widetilde{F}_{n+m+k}, \widetilde{H}_{n+k}]_{\mathcal{C}}\big) \rightarrow \widehat{\mathrm{Ext}}_{\mathcal{C}}^{m+n}(F, J) \, . \]
Because this operation is bi-additive, this yields the second Yoneda product by Proposition~\ref{prop:dirlimcommswithtensor}. Applying the isomorphism $\rho^{n+m}: \widehat{\mathcal{E}xt}_{\mathcal{C}}^{n+m}(F, J) \rightarrow BC_{\mathcal{C}}^{n+m}(F, J)$ from~\cite[Definition~6.19]{ghe24}, we deduce from the proof of~\cite[Lemma~6.20]{ghe24} that
\[ \rho^{n+m}\big(g_{\bullet+n} \circ f[n]_{\bullet+m}+\widehat{\mathrm{Null}})(F[n+m]_{\bullet}, J_{\bullet})\big) = \big(\widetilde{g}_{n+2k} \circ \widetilde{f}_{n+m+2k}+\mathcal{P}_{\mathcal{C}}(\widetilde{F}_{n+m+2k}, \widetilde{J}_{2k})\big)_{k \geq K} \]
where $K \in \mathbb{N}_0$ is chosen such that both $(f_{n+m+k})_{k \geq 2K}$ and $(g_{n+k})_{k \geq 2K}$ are chain maps. This demonstrate that the two Yoneda products agree.
\end{proof}

\section{Properties of cohomology products}\label{sec:cohomprodpropties}

This final section is dedicated to the properties of external, cup and Yoneda products. We show that these cohomology products are natural and associative. In particular, cup products turn complete cohomology and Yoneda products turn completed Ext-functors into a graded ring with identity. The canonical morphism from Ext-functors to their Mislin completion form thus ring homomorphisms. We prove that external and cup products satisfy a version of graded commutativity and that cohomology products are compatible with connecting homomorphisms. Lastly, we demonstrate that our external products generalise those for Tate--Farrell Ext-functors found in the accounts \cite[p.~110]{ben92} and~\cite[pp.~278--279]{bro82}.

\begin{conv}
In this section we assume that any module objects $A$, $B$, $C$ and $E$ satisfy the conditions of Theorem~\ref{thm:extandcupprod} so that the external products
\[ \vee: \widehat{\mathrm{Ext}}^m(A, B) \otimes \widehat{\mathrm{Ext}}^n(C, E) \rightarrow \widehat{\mathrm{Ext}}^{m+n}(A  \otimes_R C, B \otimes_R E) \]
and thus the corresponding cup products in complete cohomology exist. Regarding Yoneda products, the objects $F$, $H$, $J$ are assumed to be in an abelian category $\mathcal{C}$ with enough projectives instead.
\end{conv}

\begin{lem}\label{lem:cohomprodsnatural}
(Naturality)
\begin{enumerate}
    \item External products of completed Ext-functors and thus cup products of complete cohomology are natural. More specifically, if $r: X \rightarrow A$, $s: Y \rightarrow C$, $f: B \rightarrow M$ and $g: E \rightarrow N$ are morphisms, then the square
    \begin{center}
    \begin{tikzcd}
    \widehat{\mathrm{Ext}}^m(A{,} \, B) \otimes \widehat{\mathrm{Ext}}^n(C{,} \, E) \arrow[r, "\vee"] \arrow[d, "\widehat{\mathrm{Ext}}^m(r{,} \, f) \otimes \widehat{\mathrm{Ext}}^n(s{,} \, g)"] & \widehat{\mathrm{Ext}}^{m+n}(A \otimes_R C{,} \, B \otimes_R E) \arrow[d, "\widehat{\mathrm{Ext}}^{m+n}(r \otimes_R s{,} \, f \otimes_R g)"] \\
    \widehat{\mathrm{Ext}}^m(X{,} \, M) \otimes \widehat{\mathrm{Ext}}^n(Y{,} \, N) \arrow[r, "\vee"] & \widehat{\mathrm{Ext}}^{m+n}(X \otimes_R Y{,} \, M \otimes_R N)
    \end{tikzcd}
    \end{center}
    commutes. Therefore, cup products of complete cohomology are natural.
    \item Yoneda products of completed Ext-functors
    \[ \circ: \widehat{\mathrm{Ext}}_{\mathcal{C}}^n(H, J) \otimes \widehat{\mathrm{Ext}}_{\mathcal{C}}^m(F, H) \rightarrow \widehat{\mathrm{Ext}}_{\mathcal{C}}^{m+n}(F, J) \]
    are natural in the variable $F$ and $J$.
\end{enumerate}
\end{lem}

\begin{proof}
\textbf{(1)} To establish naturality via the resolution construction, let $f_{\bullet}: B_{\bullet} \rightarrow M_{\bullet}$, $g_{\bullet}: C_{\bullet} \rightarrow N_{\bullet}$ be lifts as in~\cite[Definition~4.2]{ghe24}. Then Diagram~\ref{diag:extprodnat}, Diagram~\ref{diag:extandconnfirst} and Diagram~\ref{diag:extandconnsecond} together with the construction of external products in proof of Theorem~\ref{thm:extandcupprod} imply the assertion. \\

\textbf{(2)} This assertion follows from~\cite[Definition~6.4]{ghe24} together with the definition of Yoneda products via the hypercohomology construction.
\end{proof}

\begin{lem}\label{lem:cohomprodsassoc}
(Associativity)
\begin{enumerate}
    \item If there are objects $A$, $B$, $C$, $E$, $F$, $H$ such that $A_{\bullet} \otimes_R C_{\bullet} \otimes_R H_{\bullet}$ is a projective resolution of $A \otimes_R C \otimes_R H$, then
    \[ \forall x \in \widehat{\mathrm{Ext}}^m(A, B), y \in \widehat{\mathrm{Ext}}^n(C, E), z \in \widehat{\mathrm{Ext}}^p(F, H): (x \vee y) \vee z = x \vee (y \vee z) \, . \]
    Therefore, cup products of complete cohomology are associative.
    \item Yoneda products of completed Ext-functors are associative.
\end{enumerate}
\end{lem}

\begin{proof}
For Yoneda products this follows from their definition through the hypercohomology construction. \\

In order to establish associativity of external products, we first need to introduce some notation. In the proof of Theorem~\ref{thm:extandcupprod} we labelled direct systems by sequences $o \in \lbrace 0, 1 \rbrace^{\mathbb{N}}$ taking the values $0$ and $1$ infinitely many times that lead to the definition of external products via the resolution construction. Now we consider sequences $o \in \lbrace 0, 1, 2 \rbrace^{\mathbb{N}}$ that take each value infinitely often. For $i \in \lbrace 0, 1, 2 \rbrace$ and $k \in \mathbb{N}$ we define $F_i(o)_k$ to be the number of times that the sequence $o$ has taken the value $i$ up to and including the $k^{\text{th}}$ term. We set $F_i(o)_0 := 0$ and define
\[ F(o) := (F_1(o)_k, F_2(o)_k, F_3(o)_k)_{k \in \mathbb{N}_0} \, . \]
One already has to use that $A_{\bullet} \otimes_R C_{\bullet} \otimes_R F_{\bullet}$ is a projective resolution of $A \otimes_R C \otimes_R H$ to conclude in Equation~\ref{eq:extprodassoc} that external products of Ext-functors are associative. Then the element $(x \vee y) \vee z$ arises from a homomorphism going from the direct limit of
\begin{equation}\label{eq:tripletensorfirst}
\big(\mathrm{Ext}^{m+F_1(r)_k}(A, \widetilde{B}_{F_1(r)_k}) \otimes \mathrm{Ext}^{n+F_2(r)_k}(C, \widetilde{E}_{F_2(r)_k}) \otimes \mathrm{Ext}^{p+F_3(r)_k}(F, \widetilde{H}_{F_3(r)_k})\big)_{k \in \mathbb{N}_0}
\end{equation}
to the direct limit of
\begin{equation}\label{eq:tripletensorsecond}
\big(\mathrm{Ext}^{m+n+p+k}(A \otimes_R C \otimes_R F, \widetilde{B}_{F_1(r)_k} \otimes_R \widetilde{E}_{F_2(r)_k} \otimes_R \widetilde{H}_{F_2(r)_k}) \big)_{k \in \mathbb{N}_0}
\end{equation}
where $r \in \lbrace 0, 1, 2 \rbrace^{\mathbb{N}}$ takes each value infinitely often. The element $x \vee (y \vee z)$ arises from an analogous homomorphism between two direct limits arising from a different sequence $s \in \lbrace 0, 1, 2 \rbrace^{\mathbb{N}}$ taking each value infinitely often. \\

In order to relate the corresponding direct systems, we tensor the terms in Diagram~\ref{diag:extprodcube} with $\mathrm{Ext}^{p+F_3(r)_k}(F, \widetilde{H}_{F_3(r)_k})$ and its corresponding identity map. If $\mathrm{Ext}^{p+F_3(r)_k}(F, \widetilde{H}_{F_3(r)_k})$ appears as the right most term in the resulting tensor  products, then no sign changes are required. If it occurs as the middle or left most term, then we need to multiply the corresponding connecting homomorphisms $\delta^{n+k+p}$ and $\delta^{n+k+p+1}$ on the left hand side with a factor of either $(-1)^{m+F_1(r)_k}$ or of $(-1)^{m+F_1(p)_k+n+F_2(r)_k}$ according to Diagram~\ref{diag:extandconnfirst} and Diagram~\ref{diag:extandconnsecond}. This brings us into the situation of the proof of~\cite[Theorem~4.15]{ghe24}. More specifically, the direct limits in Equation~\ref{eq:tripletensorfirst} and Equation~\ref{eq:tripletensorsecond} arising from $r$ and $s$ agree whenever there are infinitely many indices $k \in \mathbb{N}_0$ for which the terms of the sequences $(F(r)_k)_{k \in \mathbb{N}_0}$ and $(F(s)_k)_{k \in \mathbb{N}_0}$ coincide. \\

However, they coincide only in finitely many terms in general. Thus, it suffices to construct a sequence $u \in \lbrace 0, 1, 2 \rbrace^{\mathbb{N}}$ inductively such that $F(u)$ has infinitely many terms in common with both $F(r)$ and $F(s)$. Note that the latter sequences coincide for $k = 0$. Assume that $\kappa \in \mathbb{N}_0$ is the last index for which $F(r)_{\kappa} = F(s)_{\kappa}$ and set $u_k := r_k$ for $1 \leq k \leq \kappa$. Inductively assume that there is $K \in \mathbb{N}_0$ such that $F(u)_K = F(r)_K$. As $F(u)_K \neq F(s)_K$, there are $\lbrace h, i, j \rbrace = \lbrace 0, 1, 2 \rbrace$ such that $F_h(u) > F_h(s)$ and $F_i(u) < F_i(s)$. We apply the following procedure for $k \geq K$. If $s_k = j$, we set $u_k := j$. In case $s_k = h$, set $u_k := i$ and in case $s_k = i$, set $u_k := h$. This yields an index $K' \geq K$ for which $F_h(u)_{K'} = F_h(s)_{K'}$ or $F_i(u)_{K'} = F_i(s)_{K'}$. If this is the case for the value $h$ and we have that $F_j(u)_{K'} \neq F_j(s)_{K'}$, then we repeat the above procedure where we swap the roles of the values $h$ and $j$. This provides us an index $K'' \geq K'$ such that $F(u)_{K''} = F(s)_{K''}$. The same procedures yield an index $K''' \geq K''$ such that $F(u)_{K'''} = F(r)_{K'''}$ from which we can inductively construct the desired desired sequence $u \in \lbrace 0, 1, 2 \rbrace$.
\end{proof}

\begin{lem}\label{lem:cohomrings}
(Cohomology rings)\newline
For any object $A$ denote by $1_A \in \widehat{\mathrm{Ext}}^0(A, A)$ the element represented by the almost chain map $\mathrm{id}_{\bullet}: A_{\bullet} \rightarrow A_{\bullet}$.
\begin{enumerate}
    \item Then
    \[ \forall n \in \mathbb{Z}, x \in \widehat{H}_R^n(G, M): x \smile 1_R = x = 1_R \smile x \, . \]
    In particular, cup products turn $\bigoplus_{n \in \mathbb{Z}} \widehat{H}_R^n(G, R)$ into a graded ring with identity $1_R$.
    \item If $B$ is any other object, then
    \[ \forall n \in \mathbb{Z}, y \in \widehat{\mathrm{Ext}}_{\mathcal{C}}^n(A, B), z \in \widehat{\mathrm{Ext}}_{\mathcal{C}}^n(B, A): y \circ 1_A = y \text{ and } 1_A \circ z = z \, . \]
    In particular, Yoneda products turn $\bigoplus_{n \in \mathbb{Z}} \widehat{\mathrm{Ext}}_{\mathcal{C}}^n(A, A)$ into a graded ring with identity $1_A$.
\end{enumerate}
\end{lem}

\begin{proof}
\textbf{(1)} We prove first that $1_R$ is a unit via the resolution construction. For any element $x \in \widehat{H}_R^n(G, M)$ there is $k \in \mathbb{N}_0$ and a morphism $f: R_{n+k} \rightarrow \widetilde{M}_k$ such that the element in $H_R^{n+k}(G, \widetilde{M}_k)$ represented by $f$ is mapped to $x$ in the direct limit. By~\cite[Proposition~6.7]{ghe24} the augmentation map $\varepsilon: R_0 \rightarrow R$ gives rise to the element in $H_R^0(G, R)$ represented by $\mathrm{id}_{\bullet}: R_{\bullet} \rightarrow R_{\bullet}$ which is mapped to $1_R \in \widehat{H}_R^0(G, R)$ in the direct limit. By naturality of the isomorphisms $N \otimes_R R \cong N$, the diagram
\begin{center}
\begin{tikzcd}
    (R_{\bullet} \otimes_R R_{\bullet})_{n+k} \arrow[rrrr, "(\mathrm{id} \otimes_R \varepsilon) \circ p_{n+k+1, R_{n+k} \otimes_R R_0}"] \arrow[d, "(f \otimes_R \varepsilon) \circ p_{n+k+1, R_{n+k} \otimes_R R_0}"] & & & & R_{n+k} \otimes_R R \arrow[r, "\cong"] & R_k \arrow[d, "f"] \\
    M \otimes_R R \arrow[rrrrr, "\cong"] & & & & & M
\end{tikzcd}
\end{center}
commutes. By construction of the cup product for Ext-functors in Equation~\ref{eq:ordinaryextprod}, the left hand side gives rise to the cup product $[f] \smile 1_R \in H_R^{n+k}(G, M \otimes_R R)$ while the right hand side represents the element $[f] \in H_R^{n+k}(G, M)$. Taking direct limits, the left hand side is mapped to $x \smile 1_R \in \widehat{H}_R^n(G, M)$ while the right hand side is mapped to $x \in \widehat{H}_R^n(G, M)$. The isomorphisms on the top and bottom side result in isomorphisms of complete cohomology by functoriality. The homomorphisms
\[ (\mathrm{id}_{R_l} \otimes_R \varepsilon) \circ p_{l+1, R_l \otimes_R R_0}: (R_{\bullet} \otimes_R R_{\bullet})_l \rightarrow R_l \otimes_R R \]
form a chain equivalence of projective resolutions of $R$ as is noted in~\cite[p.~111]{bro82}. This proves that $x \smile 1_R = x$. An analogous argument demonstrates that $1_R \smile x = x$. Because connecting homomorphisms are additive and cup products of group cohomology are bi-additive according to the very end of Section~\ref{sec:ordinarycohomprods}, cup products of complete cohomology are bi-additive by passing through direct limits of the resolution construction. Because cup products of complete cohomology are associative according to Lemma~\ref{lem:cohomprodsassoc}, they turn $\bigoplus_{n \in \mathbb{Z}} \widehat{H}_R^n(G, R)$ into a graded ring with identity $1_R$. \\

\textbf{(2)}  By the hypercohomology construction, Yoneda products are bi-additive where $1_A$ is a unit. Since they are also associative according to Lemma~\ref{lem:cohomprodsassoc}, Yoneda products turn $\bigoplus_{n \in \mathbb{Z}} \widehat{\mathrm{Ext}}_{\mathcal{C}}^n(A, A)$ into a graded ring with identity $1_A$.
\end{proof}

\begin{prop}\label{prop:canisringhomom}
(Preservation by canonical morphisms)\newline
Let $\Phi^{\bullet}: \mathrm{Ext}^{\bullet}(A, -) \rightarrow \widehat{\mathrm{Ext}}^{\bullet}(A, -)$ denote the canonical morphism of the Mislin completion.
\begin{enumerate}
    \item Then $\Phi^{\bullet}$ preserves external and cup products. More specifically,
    \[ \forall x \in \mathrm{Ext}^m(A, B), y \in \mathrm{Ext}^n(C, E): \Phi^m(x) \vee \Phi^n(y) = \Phi^{m+n}(x \vee y) \, . \]
    If the ring structure derived from the cup product in Lemma~\ref{lem:cohomrings} is taken, then 
    \[ \bigoplus_{n \in \mathbb{Z}} \Phi^n: \bigoplus_{n \in \mathbb{Z}} H_R^n(G, R) \rightarrow \bigoplus_{n \in \mathbb{Z}} \widehat{H}_R^{\bullet}(G, R) \]
    is a ring homomorphism.
    \item Moreover, $\Phi^{\bullet}$ preserves Yoneda products, meaning that
    \[ \forall \xi \in \mathrm{Ext}_{\mathcal{C}}^m(H, J), \upsilon \in \mathrm{Ext}_{\mathcal{C}}^n(F, H): \Phi^n(\upsilon) \circ \Phi^m(\xi) = \Phi^{m+n}(\upsilon \circ \xi) \, . \]
    If one takes the ring structure derived from the Yoneda product, then
    \[ \bigoplus_{n \in \mathbb{Z}} \Phi^n: \bigoplus_{n \in \mathbb{Z}} \mathrm{Ext}_{\mathcal{C}}^n(F, F) \rightarrow \bigoplus_{n \in \mathbb{Z}} \widehat{\mathrm{Ext}}_{\mathcal{C}}^{\bullet}(F, F) \]
    is a ring homomorphism.
\end{enumerate}
\end{prop}

\begin{proof}
\textbf{(1)} The canonical morphism $\Phi^n: \mathrm{Ext}^n(A, -) \rightarrow \widehat{\mathrm{Ext}}^n(A, -)$ can be taken as the canonical morphism to the direct limit occurring in the resolution construction according to the proof of Proposition~\ref{prop:quotinducescan}. This together with the construction of external products via the resolution construction implies that $\Phi^{\bullet}$ preserves external and cup products. By the proof of Lemma~\ref{lem:cohomrings}, it induces the desired ring homomorphism because the identity element of $\bigoplus_{n \in \mathbb{Z}} H_R^n(G, R)$ arising from the augmentation map $\varepsilon: R_0 \rightarrow R$ is sent to the identity $1_R$ of $\bigoplus_{n \in \mathbb{Z}} \widehat{H}_R^{\bullet}(G, R)$. \\

\textbf{(2)} The definition of the Yoneda product via the hypercohomology construction from the proof of Theorem~\ref{thm:yonedaprods} is analogous to the one of Yoneda products of Ext-functors. Namely, instead of considering almost chain maps modulo chain homotopy we take chain maps modulo chain homotopy and compose them. Thus, $\Phi^{\bullet}$ preserves Yoneda products by Proposition~\ref{prop:quotinducescan}. If the element $1_F$ from Lemma~\ref{lem:cohomrings} is also taken to be a chain map modulo chain homotopy living in $\mathrm{Ext}_{\mathcal{C}}^0(F, F)$, then it represents the identity of $\bigoplus_{n \in \mathbb{Z}} \mathrm{Ext}_{\mathcal{C}}^n(F, F)$ with the ring structure derived from the Yoneda product. Because $\bigoplus_{n \in \mathbb{Z}} \Phi^n(F)$ maps the identity element to the identity element by Proposition~\ref{prop:quotinducescan}, it is a ring homomorphism.
\end{proof}

\begin{prop}\label{prop:extprodsgradedcomm}
(A form of commutativity)\newline
Assume that the tensor product $\otimes_R$ is commutative and take the homomorphism $\mathrm{swap}$ as in Diagram~\ref{diag:extprodcommity}. Then there are commutative diagrams
\begin{equation}\label{diag:swapdirlim}
\begin{tikzcd}
    {\scriptstyle \mathrm{Ext}^{m+k}(A{,} \, \widetilde{B}_k) \otimes \mathrm{Ext}^{n+k}(C{,} \, \widetilde{E}_k)} \arrow[r, "{\scriptstyle (-1)^{(m+k)(n+k)} \mathrm{swap}}"] \arrow[d, "{\scriptstyle \delta^{m+k} \otimes \mathrm{id}}"] & {\scriptstyle \mathrm{Ext}^{n+k}(C{,} \, \widetilde{E}_k) \otimes \mathrm{Ext}^{m+k}(A{,} \, \widetilde{B}_k)} \arrow[d, "{\scriptstyle (-1)^{n+k} \mathrm{id} \otimes \delta^{m+k}}"] \\
    {\scriptstyle \mathrm{Ext}^{m+k+1}(A{,} \, \widetilde{B}_{k+1}) \otimes \mathrm{Ext}^{n+k}(C{,} \, \widetilde{E}_k)} \arrow[r, shift left = 4pt, "{\scriptstyle (-1)^{(m+k+1)(n+k)} \mathrm{swap}}"] \arrow[d, "{\scriptstyle (-1)^{m+k+1} \mathrm{id} \otimes \delta^{n+k}}" near start] & {\scriptstyle \mathrm{Ext}^{n+k}(C{,} \, \widetilde{E}_k) \otimes \mathrm{Ext}^{m+k+1}(A{,} \, \widetilde{B}_{k+1})} \arrow[d, "{\scriptstyle \delta^{n+k} \otimes \mathrm{id}}"] \\
    {\scriptstyle \mathrm{Ext}^{m+k+1}(A{,} \, \widetilde{B}_{k+1}) \otimes \mathrm{Ext}^{n+k+1}(C{,} \, \widetilde{E}_{k+1})} \arrow[r, shift left = 4pt, "{\scriptstyle (-1)^{(m+k+1)(n+k+1)} \mathrm{swap}}"] & {\scriptstyle \mathrm{Ext}^{n+k+1}(C{,} \, \widetilde{E}_{k+1}) \otimes \mathrm{Ext}_R^{m+k+1}(A{,} \, \widetilde{B}_{k+1})}
\end{tikzcd}
\end{equation}
The homomorphism in their direct limit
\[ \widehat{\mathrm{swap}}: \widehat{\mathrm{Ext}}^m(A, B) \otimes \widehat{\mathrm{Ext}}^n(C, E) \rightarrow \widehat{\mathrm{Ext}}^n(C, E) \otimes \widehat{\mathrm{Ext}}^m(A, B) \]
renders external products commutative in the sense that the diagram
\begin{equation}\label{diag:fakecommity}
\begin{tikzcd}
    \widehat{\mathrm{Ext}}^m(A{,} \, B) \otimes \widehat{\mathrm{Ext}}^n(C{,} \, E) \arrow[r, "\vee"] \arrow[d, "\widehat{\mathrm{swap}}"] & \widehat{\mathrm{Ext}}^{m+n}(A \otimes_R C{,} \, B \otimes_R E) \arrow[d, "\widehat{\mathrm{Ext}}^{m+n}(\mathrm{swap}{,} \, \mathrm{swap})"] \\
    \widehat{\mathrm{Ext}}^n(C{,} \, E) \otimes \widehat{\mathrm{Ext}}^m(A{,} \, B) \arrow[r, "\vee"] & \widehat{\mathrm{Ext}}^{m+n}(B \otimes_R E{,} \, A \otimes_R C)
\end{tikzcd}
\end{equation}
commutes. Cup products of complete cohomology satisfy the same form of commutativity.
\end{prop}

\begin{proof}
Diagram~\ref{diag:swapdirlim} commutes by definition. If we apply Proposition~\ref{prop:dirlimcommswithtensor}, the homomorphisms on the left hand side of Diagram~\ref{diag:swapdirlim} do not immediately yield $\widehat{\mathrm{Ext}}^m(A, B) \otimes \widehat{\mathrm{Ext}}^n(C, E)$ in the direct limit as in resolution construction. However, if we concatenate four copies of Diagram~\ref{diag:swapdirlim}, the sign issues vanish yielding the direct limit of the desired form. For the same reason the homomorphisms on the right hand side result in the correct direct limit from which we obtain the homomorphism $\widehat{\mathrm{swap}}$. Next, we form with Diagram~\ref{diag:extprodcommity} the following direct system of commuting squares. We connect its left hand homomorphisms via Diagram~\ref{diag:swapdirlim} to each other. We connect its top and its bottom homomorphisms via Diagram~\ref{diag:extandconnfirst} and Diagram~\ref{diag:extandconnsecond} where the factor of $(-1)^m$ is assigned to the term $\mathrm{id} \otimes \delta^n$ instead of the term $\delta^{m+n}$ in the latter diagram. According to Equation~\ref{eq:extprodgradedcomm}, the right hand morphisms are of the form $\mathrm{Ext}^{m+n+2k}(\mathrm{swap}, \mathrm{swap})$ and connected by connecting homomorphism. Then Diagram~\ref{diag:fakecommity} results as the direct limit of these squares.
\end{proof}

\begin{lem}\label{lem:cohomprodsandconn}
(Relations with connecting homomorphisms)
\begin{enumerate}
    \item Let $0 \rightarrow B \rightarrow B' \rightarrow B'' \rightarrow 0$ and $0 \rightarrow E \rightarrow E' \rightarrow E'' \rightarrow 0$ be short exact sequences in $Mod$. Then the diagrams
    \begin{equation}\label{diag:complextandconnfirst}
    \begin{tikzcd}
        \widehat{\mathrm{Ext}}^m(A{,} \, B) \otimes \widehat{\mathrm{Ext}}^n(C{,} \, E) \arrow[r, "\vee"] \arrow[d, "\widehat{\delta}^m \otimes \mathrm{id}"] & \widehat{\mathrm{Ext}}^{m+n}(A \otimes_R C{,} \, B \otimes_R E) \arrow[d, "\widehat{\delta}^{m+n}"] \\
        \widehat{\mathrm{Ext}}^{m+1}(A{,} \, B'') \otimes \widehat{\mathrm{Ext}}^n(C{,} \, E) \arrow[r, "\vee"] & \widehat{\mathrm{Ext}}^{m+n+1}(A \otimes_R C{,} \, B'' \otimes_R E)
    \end{tikzcd}
    \end{equation}
    and
    \begin{equation}\label{diag:complextandconnsecond}
    \begin{tikzcd}
        \widehat{\mathrm{Ext}}^m(A{,} \, B) \otimes \widehat{\mathrm{Ext}}^n(C{,} \, E) \arrow[r, "\vee"] \arrow[d, "\mathrm{id} \otimes \widehat{\delta}^n"] & \widehat{\mathrm{Ext}}^{m+n}(A \otimes_R C{,} \, B \otimes_R E) \arrow[d, "(-1)^m \widehat{\delta}^{m+n}"] \\
        \widehat{\mathrm{Ext}}^m(A{,} \, B) \otimes \widehat{\mathrm{Ext}}^{n+1}(C{,} \, E'') \arrow[r, "\vee"] & \widehat{\mathrm{Ext}}^{m+n+1}(A \otimes_R C{,} \, B \otimes_R E'')
    \end{tikzcd}
    \end{equation}
    commute. Cup products of complete cohomology satisfy the same relations with connecting homomorphisms.
    \item For any $F, H, J \in \mathrm{obj}(\mathcal{C})$ and any short exact sequence $0 \rightarrow J \rightarrow J' \rightarrow J'' \rightarrow 0$ the diagram
    \begin{center}
    \begin{tikzcd}
        \widehat{\mathrm{Ext}}_{\mathcal{C}}^n(H{,} \, J) \otimes \widehat{\mathrm{Ext}}_{\mathcal{C}}^m(F{,} \, H) \arrow[r, "\circ"] \arrow[d, "\widehat{\delta}^n \circ \mathrm{id}"] & \widehat{\mathrm{Ext}}_{\mathcal{C}}^{m+n}(F{,} \, J) \arrow[d, "\widehat{\delta}^{m+n}"] \\
        \widehat{\mathrm{Ext}}_{\mathcal{C}}^{n+1}(H{,} \, J'') \otimes \widehat{\mathrm{Ext}}_{\mathcal{C}}^m(F{,} \, H) \arrow[r, "\circ"] & \widehat{\mathrm{Ext}}_{\mathcal{C}}^{m+n+1}(F{,} \, J'')
    \end{tikzcd}
    \end{center}
    commutes.
\end{enumerate}
\end{lem}

\begin{proof}
The statement about Yoneda products follows their definition via the hypercohomology construction and the construction of the connecting homomorphism found in~\cite[Definition~6.6]{ghe24}. Let us argue that Diagram~\ref{diag:complextandconnfirst} commutes. For $k, l \in \mathbb{N}_0$ set $K := n+k$ and $L := n+l$. It follows from Diagram~\ref{diag:extprodcube} and the proof of Theorem~\ref{thm:extandcupprod} that the diagram on the next page is commutative. The left hand side of Diagram~\ref{diag:extproddoublecube} gives rise to the left hand homomorphisms of Diagram~\ref{diag:complextandconnfirst}, the front side to the bottom homomorphism and the  back side to the top homomorphism. After concatenating four copies of Diagram~\ref{diag:extproddoublecube}, the resulting direct system from the right hand side gives rise to the right hand homomorphism in Diagram~\ref{diag:complextandconnfirst}. One can use Diagram~\ref{diag:extprodcube} and the proof of Theorem~\ref{thm:extandcupprod} to construct an analogous diagram which gives rise to a direct system of commuting squares in whose direct limit we obtain Diagram~\ref{diag:complextandconnsecond}.
\end{proof}\pagebreak

\begin{samepage}
\begin{center}
\rotatebox{90}{%
\begin{tikzcd}[ampersand replacement = \&, column sep = tiny]
  \& {\scriptstyle \mathrm{Ext}_R^K(A{,} \, \widetilde{B}_k) \otimes \mathrm{Ext}^L(C{,} \, \widetilde{E}_l)} \arrow[dl, "{\scriptstyle (-1)^k \delta^K \otimes \mathrm{id}}"] \arrow[rr, "{\scriptstyle \vee}" near end] \arrow[ddd, "{\scriptstyle \mathrm{id} \otimes \delta^L}" near start]
    \& \& {\scriptstyle \mathrm{Ext}^{K+L}(A \otimes_R C{,} \, \widetilde{B}_k \otimes_R \widetilde{E}_l)} \arrow[dl, "{\scriptstyle (-1)^k \delta^{K+L}}"] \arrow[ddd, "{\scriptstyle (-1)^K \delta^{K+L}}" near start] \\
  {\scriptstyle \mathrm{Ext}^{K+1}(A{,} \, \widetilde{B}_{k+1}'') \otimes \mathrm{Ext}^L(C{,} \, \widetilde{E}_l)} \arrow[rr, crossing over, "{\scriptstyle \vee}" near end] \arrow[ddd, "{\scriptstyle \mathrm{id} \otimes \delta^L}" near start]
    \& \& {\scriptstyle \mathrm{Ext}^{K+L+1}(A \otimes_R C{,} \, \widetilde{B}_{k+1}'' \otimes_R \widetilde{E}_l)} \\ \\
  \& {\scriptstyle \mathrm{Ext}^K(A{,} \, \widetilde{B}_k) \otimes \mathrm{Ext}^{L+1}(C{,} \, \widetilde{E}_{l+1})} \arrow[dl, "{\scriptstyle (-1)^k \delta^K \otimes \mathrm{id}}" near start] \arrow[rr, "{\scriptstyle \vee}" near end] \arrow[ddd, "{\scriptstyle \delta^K \otimes \mathrm{id}}" near start]
    \& \& {\scriptstyle \mathrm{Ext}^{K+L+1}(A \otimes_R C{,} \, \widetilde{B}_k \otimes_R \widetilde{E}_{l+1})} \arrow[dl, "{\scriptstyle (-1)^k \delta^{K+L+1}}"] \arrow[ddd, "{\scriptstyle \delta^{K+L+1}}" near start] \\
  {\scriptstyle \mathrm{Ext}^{K+1}(A{,} \, \widetilde{B}_{k+1}'') \otimes \mathrm{Ext}^{L+1}(C{,} \, \widetilde{E}_{l+1})} \arrow[rr, crossing over, "{\scriptstyle \vee}" near end] \arrow[ddd, "{\scriptstyle \delta^{K+1} \otimes \mathrm{id}}" near start]
    \& \& {\scriptstyle \mathrm{Ext}^{K+L+2}(A \otimes_R C{,} \, \widetilde{B}_{k+1}'' \otimes_R \widetilde{E}_{l+1})} \arrow[from=uuu, crossing over, "{\scriptstyle (-1)^{K+1} \delta^{K+L+1}}" near start] \\ \\
    \& {\scriptstyle \mathrm{Ext}^{K+1}(A{,} \, \widetilde{B}_{k+1}) \otimes \mathrm{Ext}^{L+1}(C{,} \, \widetilde{E}_{l+1})} \arrow[dl, "{\scriptstyle (-1)^{k+1}\delta^{K+1} \otimes \mathrm{id}}" near start] \arrow[rr, "{\scriptstyle \vee}" near end]
    \& \& {\scriptstyle \mathrm{Ext}^{K+L+2}(A \otimes_R C{,} \, \widetilde{B}_{k+1} \otimes_R \widetilde{E}_{l+1})} \arrow[dl, "{\scriptstyle (-1)^{k+1}\delta^{K+L+2}}"] \\
  {\scriptstyle \mathrm{Ext}^{K+2}(A{,} \, \widetilde{B}_{k+2}'') \otimes \mathrm{Ext}^{L+1}(C{,} \, \widetilde{E}_{l+1})} \arrow[rr, "{\scriptstyle \vee}" near end]
    \& \& {\scriptstyle \mathrm{Ext}^{K+L+3}(A \otimes_R C{,} \, \widetilde{B}_{k+2}'' \otimes_R \widetilde{E}_{l+1})} \arrow[from=uuu, crossing over, "{\scriptstyle \delta^{K+L+2}}" near start] \\
\end{tikzcd}
}
\end{center}
{\vspace{-110mm}
\begin{equation}\label{diag:extproddoublecube}
{} \quad {}
\end{equation}
${} \quad {}$ \vspace{100mm}}
\end{samepage}

Our external products and cup products generalise previous constructions because they do do not hinge on dimension shifting in order to be well defined. More specifically, cup products have been developed for Tate--Farrell cohomology of groups with finite virtual cohomological dimension in~\cite[pp.~278--279]{bro82}. External products for completed unenriched Ext-functors of any category of modules over a ring are constructed via the naïve construction in~\cite[p.~110]{ben92}. However, the latter external products are only well defined for Tate--Farrell Ext-functors. As is explained in~\cite[pp.~278--279]{bro82}, this is because both external products can be obtained from the external products of `ordinary' Ext-functors via dimension shifiting in sufficiently high dimensions. Therefore, the following generalises both \cite[p.~110]{ben92} and~\cite[pp.~278--279]{bro82}.

\begin{lem}\label{lem:tatefarrelltwo}
(Generalising Tate--Farrell, again)\newline
External products of completed Ext-functors generalise external products of Tate--Farrell Ext-functors. Accordingly, cup products of complete cohomology generalise cup products of Tate--Farrell cohomology.
\end{lem}

\begin{proof}
Let $A$ be an object admitting a complete resolution. By the definition of Tate--Farrell Ext-functors, Proposition~\ref{prop:detamislincompletion} and Lemma~\ref{lem:tatefarrell}, there is $k \in \mathbb{N}$ such that for every $n \geq k$ the canonical morphism $\Phi^n: \mathrm{Ext}^n(A, B) \rightarrow \widehat{\mathrm{Ext}}^n(A, B)$ is the identity map. Thus, external products of Ext-functors agree with the external products of completed Ext-functors in degrees higher or equal to $k$ by Proposition~\ref{prop:canisringhomom}. These uniquely determine external products of any other degree due to dimension shifting (Theorem~\ref{thm:dimshifting}) and the fact the external products commute with connecting homomorphisms (Lemma~\ref{lem:cohomprodsandconn}), which completes the proof.
\end{proof}

\section*{Acknowledgements}

A special thanks goes to Peter H. Kropholler for sharing his expertise with me. He has discussed examples of complete cohomology groups of discrete groups with me and pointed out his joint paper with Jonathan Cornick ``On Complete Resolutions''. In particular, he has advised me to generalise Yoneda and external products to complete cohomology and introduced me to condensed mathematics. \\

Moreover, I am thankful to Andrew Fisher for pointing out the paper ``Complete Homology over Associative Rings'' by Olgur Celikbas et al. Lastly, I would like to acknowledge Alejandro Adem, Jon F.\! Carlson, Joel Friedman, Kalle Karu, Zinovy Reichstein and Ben Williams who assessed my PhD thesis and whose invaluable suggestions have led to an improvement of the exposition.

\addcontentsline{toc}{section}{References}
\renewcommand{\bibname}{References}
\bibliographystyle{plain}  
\bibliography{refs}        

\end{document}